\documentclass[reqno]{amsart}
\usepackage{amssymb}
\usepackage[dvips]{epsfig}
\usepackage{graphicx}
\usepackage{color}
\usepackage{amsmath}
\usepackage{amssymb}
\usepackage{amsfonts}
\usepackage{amsthm}
\usepackage{bbm}
\usepackage{mathrsfs}

\newtheorem{thm}{Theorem}[section]
\newtheorem{lem}[thm]{Lemma}

\theoremstyle{remark}
\newtheorem{rem}{\bf Remark}[section]
\theoremstyle{definition}
\newtheorem{defn}[thm]{Definition}

\numberwithin{equation}{section}

\usepackage[colorlinks=true,pdfstartview=FitV,linkcolor=magenta,citecolor=cyan]{hyperref}
\usepackage{bm}
\begin{document}
\title[Full dimensional KAM tori for NLS] %Use the shortened version of the full title
      {On the existence of full dimensional KAM tori for 1D periodic nonlinear Schr\"odinger equation }

\author[Y.Wu]{Yuan Wu}
\address[Y. Wu]{School of Mathematics and Statistics,
Huazhong University of Science and Technology, Wuhan 430070, China}
\address[Y. Wu]{Hubei Key Laboratory of Engineering Modeling and Scientific Computing,
Huazhong University of Science and Technology, Wuhan 430074, China}
\email{wuyuan@hust.edu.cn}

\date{\today}


%%% ----------------------------------------------------------------------

\begin{abstract}
In this paper, we will prove the existence of full dimensional tori for 1-dimensional nonlinear Schr\"odinger equation
\begin{eqnarray}\label{maineq0}
\mathbf{i}u_{t}-u_{xx}+V*u+\epsilon f(x)|u|^{4}u=0,\ x\in\mathbb{T}=\mathbb{R}/2\pi\mathbb{Z},
\end{eqnarray}
with boundary conditions, where $V*$ is the Fourier multiplier, and $f(x)$ is  Gevrey smooth. Here the radius of the invariant tori satisfies a slower decay, i.e.
\[
I_n\sim e^{-2\ln^{\sigma}|n|}, \mbox{as}\ n\rightarrow\infty,
\]
for any $ \sigma> 2, $ which extends results of Bourgain \cite{BJFA2005}  and Cong \cite{cong2024} to the case that the nonlinear perturbation depends explicitly on the space variable $x$.
\end{abstract}

\maketitle

%The title of your section 1
\section{Introduction and main result}
%Consider a nearly-integrable Hamiltonian systems of $n$-freedom
%\begin{equation}\label{yuan1}H=H_0(I)+\epsilon H_{1}(\theta,I),\end{equation}
%with the standard symplectic  structure $\mathrm{d}\theta\wedge\mathrm{d}I$ on $\mathbb{T}^n\times\mathbb{R}^{n}$ and the angle-action variable $(\theta,I)$ belongs to some domain
%$\mathbb{T}^n\times D\subseteq\mathbb{T}^n\times\mathbb{R}^{n}$. Assume the unperturbed Hamiltonian  $H_0(I)$ is independent of $\theta$ and satisfies Kolmogorov non-degenerate condition
%$$\mathrm{det}(\partial^2H_0(I))\neq0,\ I\in D.$$
%Also assume $H_0,H_1$ are smooth sufficiently. Then the well-known Kolmogorov-Arnold-Moser (KAM) theorem (\cite{A1961,K1954,M1962}) claims that any invariant tori of the unperturbed $H_0$ with prescribed Diophantine frequency $\omega(I_0)=\frac{\partial{H_0}(I_0)}{\partial I}$ for some $I_0\in D$ persist under a small perturbation $\epsilon H_1(I,\theta)$. This theorem is now called the classical KAM theorem and the persisted tori %called full dimensional
%KAM tori.
The study of the full dimensional invariant tori for Hamiltonian PDEs has attracted many attentions over the years, cf. e.g., \cite{BBM2014,BBM2016,BA1998,BZ2004,CLSY,G2012,GX2013,NG2007,P1990,WG2011}.\
In this paper, we focus on  the nonlinear Schr\"odinger equation (NLS) with periodic boundary conditions
\begin{equation}\label{maineq}
\textbf{i}u_{t}- u_{xx}+V*u+\epsilon f(x)|u|^4u=0,\  x\in\mathbb{T},
\end{equation}
where $\textbf{i}=\sqrt{-1}$, $V*$ is a Fourier multiplier defined by
\begin{equation}\nonumber
V*u=\sum_{n\in\mathbb{Z}}V_n\widehat{u}_ne^{\textbf{i}nx},\  V_n\in[-1,1],
\end{equation} $f(x)$ is $2\pi$-periodic and gevrey analytic in $x$.\ Note that the basic idea in \cite{B1996,P2002} is to use repeatedly (infinitely many times) the KAM theorem dealing with lower dimensional KAM tori. In a different way, Bourgain \cite{BJFA2005} constructed the full dimensional tori directly, and then was generalized to a slower decay rate by Cong \cite{cong2024}. We also mention that P\"{o}schel in \cite{P2002} proved the existence of full dimensional tori for infinite dimensional Hamiltonian system with spatial structure of short range couplings. The present work aims to prove the existence of the full dimensional tori for such a family of NLS (\ref{maineq}) via classical KAM way in the spirits of Bourgain \cite{BJFA2005} and  Cong \cite{cong2024}.

The groundbreaking work of Bourgain \cite{BJFA2005} was to treat all Fourier modes at once under Diophantine conditions. See the nonresonant conditions (\ref{005}) for the details, which is similar as the one given in \cite{BJFA2005}. It is well known that the core of KAM theory is how to
deal with small divisor. Note that the conditions (\ref{005}) is totally different from the nonresonant
conditions used to construct the low dimensional tori, since the factors $ n^{4}$ appears in the
denominator, which causes a much worse small denominator problem. Besides this, Bourgain constructed
a KAM theorem in Cartesian coordinates while all previous researchers \cite{CW1993,KPB2003,K1987,KB2000,KZ2004} studied in action-angle coordinates.
There are some differences between two coordinates. See \cite{BFN2020} for more comments. This method is based on quantitative analysis of small denominator problem, requiring the momentum conservation and the quadratic growth of the frequencies for 1-dimensional NLS. More precisely, Bourgain made use of two simple but important facts: let $(n_i)$ be a finite set of modes satisfying
$$|n_1|\geq |n_2|\geq \cdots$$
and
\begin{equation}\label{M}
n_1-n_2+n_3-\cdots=0.
\end{equation}
The first observation is that: the first two biggest indices $|n_1|$ and $|n_2|$ can be controlled by other indices unless $n_1=n_2$, i.e.
\begin{equation}\label{M''}
|n_1|+|n_2|\leq \sum_{j\geq 3}|n_j|,
\end{equation}
under the assumption that
\begin{equation}\label{M'}
n^2_1
-n^2_2
+n^2_3
-\cdots=o(1).
\end{equation}
Another key observation is
\begin{equation}\label{M'''}
\sum_{j\geq 1}\sqrt{|n_j|}-2\sqrt{|n_1|}\geq \frac{1}{4}\sum_{j\geq 3}\sqrt{|n_j|}.
\end{equation}
The conditions (\ref{M}) and (\ref{M'}) hold true for 1-dimensional NLS, and the equality (\ref{M'''}) guarantees the KAM iteration works. More than ten years later, Cong-Liu-Shi-Yuan \cite{CLSY} extended Bourgain's results to any $\theta\in (0,1)$ and proved the obtained tori are stable in a sub-exponential long time. In 2021, Biasco-Massetti-Procesi \cite{BMP2021} proved the
existence and linear stability of almost periodic solution for 1-dimensional NLS, which is based on a more geometric point of view.\ Recently, Cong \cite{cong2024} generalized Bourgain's results to a slower decay rate, where the action satisfying $I_n\sim e^{-2\ln^{\sigma}|n|}, \sigma> 2.$ Instead of (\ref{M'''}), Cong  \cite{cong2024} proved another equality
\begin{equation}\label{M''''}
\ln^{\sigma}(x+y)-\ln^{\sigma}x-\frac12\ln^{\sigma}y\leq 0, \quad \mbox{for}\ c(\sigma)\leq y\leq x,
\end{equation}
where $ c(\sigma)$ is a positive constant depending $ \sigma$ only.

A natural question arises: can one establish the full dimensional invariant tori be with a suitable decay without conditions (\ref{M}) or (\ref{M'})? This problem has been solved by Cong-Mi-Shi-Wu \cite{CMSW} and Cong-Yuan \cite{CY2021} with a sub-exponential decay of the action. In this paper, we will discuss the existence of full dimensional KAM tori for equations (\ref{maineq}) with decay rate of \begin{equation}\label{083102}\frac{1}{4}e^{-2r{\ln^{\sigma}|n|}}\leq I_n\leq 4e^{-2r{\ln^{\sigma}|n|}}, n\in\mathbb{Z},\ r>0, \ \sigma>2.
\end{equation}
Written in Fourier modes $ (q_{n})_{n\in \mathbb{Z}}$,\ then (\ref{maineq}) can be rewritten as
\begin{eqnarray}
\label{002} \dot{q}_{n}= \mathbf{i}\frac{\partial H}{\partial \overline{q}_{n}}
\end{eqnarray}
with the Hamiltonian
\begin{equation}
\label{003}
H(q,\overline{q})=\sum_{n\in\mathbb{Z}}(n^{2}+V_{n})|q_{n}|^{2}
+\epsilon\sum_{n\in\mathbb{Z}}\underset{n_{1}-n_{2}+n_{3}-n_{4}+n_{5}-n_{6}=-n}{\sum}
\widehat{f}(n)q_{n_{1}}\overline{q}_{n_{2}}q_{n_{3}}\overline{q}_{n_{4}}q_{n_{5}}\overline{q}_{n_{6}}.
\end{equation}
Note that the condition (\ref{M}) is no longer valid for the Hamiltonian (\ref{003}).\ But if the function $f(x)$ is Gevrey smooth with $\mu>0$, i.e.
\begin{align}\label{M''}
|\widehat{f}(n)|\leq Ce^{-\mu\ln^{\sigma}|n|},\ \mu>0,\ \sigma>2,
\end{align}
then $|n_1|+|n_2|$ can be controlled by $\sum_{j\geq 3}|n_j|+|n|$ and the property (\ref{M''}) can also guarantee the KAM iteration works.

To state our result precisely, we will give some definitions firstly.
\begin{defn}
Denote $ \|x\| = \mathrm{dist} (x,\mathbb{Z})$. A vector $\omega=(\omega_n)_{n\in\mathbb{Z}}$ is called to be Diophantine if there exists a real number $ \gamma > 0 $ such that the following resonance issues
\begin{eqnarray}
\label{005} \left\| \sum_{n\in \mathbb{Z}}l_{n}\omega_{n}\right\|\geq \gamma \prod_{n\in \mathbb{Z}}\frac{1}{1+l^{2}_{n}\langle n\rangle^{4}}
\end{eqnarray}
hold, where $ 0 \neq l = (l_{n})_{n\in \mathbb{Z}}$ is a finitely supported sequence of integers and
\[
\langle n\rangle= \max\{ 1,|n|\}.
\]
\end{defn}
\begin{thm}\label{Thm}
Given $r>0$, $\sigma > 2$ and a Diophantine vector $ \omega = (\omega_{n})_{n\in\mathbb{Z}}$ satisfying $
\sup_{n}|\omega_{n}| < 1$,\ then for any $\mu>2r$, sufficiently small $ \epsilon > 0 $ and some appropriate $V$, (\ref{maineq}) has a full dimensional invariant torus $ \mathcal{E}$ with amplitude in $ \mathfrak{H}_{r,\infty}$ satisfying:\\
\begin{itemize}
\item[(1)] the amplitude of $ \mathcal{E}$ is restricted as
\begin{equation*}
\frac14e^{-2r\ln^{\sigma}\lfloor n\rfloor}\leq I_n\leq 4e^{-2r\ln^{\sigma}\lfloor n\rfloor}, \forall n,
\end{equation*}
where
$$
\lfloor n\rfloor=\max\{c(\sigma), |n|\};
$$
\item[(2)] the frequency on  $ \mathcal{E}$ was prescribed to be $(n^2+\omega_n)_{n\in\mathbb{Z}}$;\\
\item[(3)] the invariant torus  $ \mathcal{E}$ is linearly stable.
\end{itemize}
\end{thm}

\section{The norm of the Hamiltonian}
 Let $q=(q_n)_{n\in\mathbb{Z}}$ and its complex conjugate $\bar {q}=(\bar q_n)_{n\in\mathbb{Z}}$. Introduce $I_n=|q_n|^2$ and $J_n=I_n-I_n(0)$ as notations but not as new variables, where $I_n(0)$ will be considered as the initial data. Then the  Hamiltonian (\ref{maineq}) has the form of
\begin{equation*}
H(q,\bar q)=N(q,\bar q)+R(q,\bar q),
\end{equation*}
where
\begin{equation*}
N(q,\bar q)=\sum_{n\in\mathbb{Z}}(n^2+V_n)|q_n|^2,
\end{equation*}
 \begin{equation*}
 R(q,\bar q)=\sum_{a,k,k'\in{\mathbb{N}^{\mathbb{Z}}}}B_{akk'}\mathcal{M}_{akk'}
 \end{equation*}
with
\begin{eqnarray}\nonumber
\mathcal{M}_{akk'}=\prod_{n\in\mathbb{Z}}I_n(0)^{a_n}q_n^{k_n}\bar q_n^{k_n'},
\end{eqnarray}
and $B_{akk'}$ are the coefficients.

Define by
\begin{eqnarray}\label{008}
\mbox{supp}\ \mathcal{M}_{akk'}=\{n:2a_n+k_n+k_n'\neq 0\},
\end{eqnarray}
and define the momentum of $\mathcal{M}_{akk'}$ by
\begin{eqnarray}\label{010}
\mbox{momentum}\ \mathcal{M}_{akk'}:=m(k,k')=\sum_{n\in\mathbb{Z}} (k_n-k_n')n.
\end{eqnarray}
Moreover, denote by
\[
 n^{\ast}_{1}= \max\{|n|: a_{n}+k_{n}+k'_{n} \neq 0\},
\]
and
\[
m^*(k,k')=\left|m(k,k')\right|.
\]
Now we define the norm of the Hamiltonian as follows
\begin{defn}\label{007}
 Given $ \sigma> 2$ and $ r > 0 $,\ we define the Banach space $ \mathfrak{H}_{r,\infty}$  consisting of all complex sequences $ q = (q_{n})_{n\in\mathbb{Z}}$ with
\begin{eqnarray}
\label{H0} \|q \|_{r,\infty}= \sup_{n\in \mathbb{Z}}|q_{n}|e^{r\ln^{\sigma}\lfloor n\rfloor} < \infty.
\end{eqnarray}
\end{defn}

\begin{defn}
For any given $ \rho > 0,\mu > 0$ and $\sigma > 2$, define the norm of the Hamiltonian
$R$ by
\begin{eqnarray}\label{H00}
\| R\|_{\rho,\mu}= \sup_{a,k,k'\in\mathbb{N}^{\mathbb{Z}}}\frac{\left|B_{akk'}\right|}{e^{\rho\sum_{n}(2a_{n}+k_{n}+k_{n}')
\ln^{\sigma}\lfloor n\rfloor
-2\rho\ln^{\sigma}\lfloor n_1^\ast\rfloor-\mu \ln^{\sigma}\lfloor m^\ast(k,k')\rfloor}}.
\end{eqnarray}
\end{defn}

For any $a, k, k\in \mathbb{N}^{\mathbb{Z}}$, denote $(n^{*}_{i})_{i\geq1}$ the decreasing rearrangement of
\[
\{|n| : \mbox{where}\ n\ \mbox{is repeated}\ 2a_{n} + k_{n} + k'_{n}\ \mbox{times}\},
\]
and $ (n_{i})_{i\geq1}$ the system
\[
\{n : \mbox{where}\ n\ \mbox{is repeated}\ 2a_{n} + k_{n} + k'_{n}\ \mbox{times}\},
\]
which satisfies $|n_{1}|\geq|n_{2}|\geq\cdot\cdot\cdot.$

In fact we can prove a positive lower bound of
\[
\sum_{n}(2a_{n}+k_{n}+k_{n}')\ln^{\sigma}\lfloor n\rfloor
-2\ln^{\sigma}\lfloor n_1^\ast\rfloor+\ln^{\sigma}\lfloor m^*{(k,k')}\rfloor,
\]
which is important to overcome the small divisor. Precisely we have the following
lemma:

\begin{lem}\label{H1}
Denote $(n^*_i)_{i\geq1}$ the decreasing rearrangement of
\begin{equation*}
\{|n|:\ \mbox{where $n$ is repeated}\ 2a_n+k_n+k_n'\ \mbox{times}\}.
\end{equation*}
Then for any $ \sigma> 2$, one has
\begin{eqnarray}\label{H2}
\sum_{n\in\mathbb{Z}}(2a_n+k_n+k_n')\ln^{\sigma}\lfloor n\rfloor-2\ln^{\sigma}\lfloor n_1^*\rfloor +\ln^{\sigma}\lfloor m^*{(k,k')}\rfloor\geq\frac{1}{2}\left(\sum_{i\geq 3}\ln^{\sigma}\lfloor n_i^*\rfloor\right).
\end{eqnarray}
\end{lem}
\begin{proof}  Without loss of generality,\ denote $(n_i)_{i\geq 1},\ |n_1|\geq |n_2|\geq\cdots$, the system $\{\mbox{$ n $ is repeated}\ 2a_n+k_n+k_n'\ \mbox{times}\}$ and we have $n_i^*=|n_i|\ \mbox{for}\ \forall\ i\geq1$. There exists $(\mu_i)_{i\geq1}$ with $\mu_i\in\{-1,1\}$ such that
\begin{equation*}
m(k,k')=\sum_{i\geq1}\mu_in_i,
\end{equation*}
and hence
\begin{equation*}
n_1^*\leq \sum_{i\geq 2}n_i^*+m^*(k,k').
\end{equation*}
Consequently
\begin{equation*}
\ln^{\sigma}\lfloor n_1^*\rfloor\leq\ln^{\sigma}\left(\sum_{i\geq 2}\lfloor n_i^*\rfloor+\lfloor m^*(k,k')\rfloor\right).
\end{equation*}
Thus the inequality (\ref{H2}) will follow from the inequality
\begin{equation}\label{0001}
\sum_{i\geq 2}\ln^{\sigma}\lfloor n_i^*\rfloor+ \ln^{\sigma}\lfloor m^*(k,k')\rfloor\geq \ln^{\sigma} \left(\sum_{i\geq 2}\lfloor n_i^*\rfloor+\lfloor m^*(k,k')\rfloor\right)+\frac{1}{2}\left(\sum_{i\geq 3}\ln^{\sigma}\lfloor n_i^*\rfloor\right).
\end{equation}
By iteration and in view of (\ref{122201}), one obtains
\begin{eqnarray*}
&&\sum_{i\geq 2}\ln^{\sigma}\lfloor n_i^*\rfloor+\ln^{\sigma}\lfloor m^*(k,k')\rfloor\\
&\geq&\ln^{\sigma}\left(\sum_{i\geq 2}\lfloor n_i^*\rfloor\right)
+\ln^{\sigma}\lfloor m^*(k,k')\rfloor+\frac{1}{2}\left(\sum_{i\geq 3}\ln^{\sigma}\lfloor n_i^*\rfloor\right)\\
&\geq&\ln^{\sigma}\left(\sum_{i\geq 2}\lfloor n_i^*\rfloor+\lfloor m^*(k,k')\rfloor\right)+ \frac{1}{2}\ln^{\sigma}\left(\sum_{i\geq 3}\lfloor n_i^*\rfloor\right),
\end{eqnarray*}
which finishes the proof of (\ref{H2}).
\end{proof}

\begin{lem}\label{H3}\textbf{(Poisson Bracket)}
Let $ \sigma > 2,\rho,\mu>0$ and
$$0<\delta_{1},\delta_{2}< \min\{\frac{1}{4}\rho, 3-2\sqrt{2}\}.$$
Then one has
\begin{equation}\label{H4}
\left|\left|\{H_1,H_2\}\right|\right|_{\rho,\mu}\leq \frac{1}{\delta_{2}}\exp\left\{\frac{300}{\delta_{1}}\exp\left\{\left(\frac{50}{\delta_{1}}
\right)^{\frac{1}{\sigma-1}}\right\}\right\}
||H_1||_{\rho-\delta_{1},\mu+2\delta_{1}}||H_2||_{\rho-\delta_{2},\mu+2\delta_{2}}.
\end{equation}
\end{lem}
\begin{proof}
Let
\begin{equation*}
H_1=\sum_{a,k,k'} b_{akk'}\mathcal{M}_{akk'}
\end{equation*}
and
\begin{equation*}
H_2=\sum_{A,K,K'} B_{AKK'}\mathcal{M}_{AKK'}.
\end{equation*}
It follows easily that
\begin{equation*}
\{H_1,H_2\}=\sum_{a,k,k',A,K,K'}b_{akk'}B_{AKK'}\{\mathcal{M}_{akk'},\mathcal{M}_{AKK'}\},
\end{equation*}
where
\begin{eqnarray*}
\{\mathcal{M}_{akk'},\mathcal{M}_{AKK'}\}
&=&\frac{1}{2\textbf{i}}\sum_j\left(\prod_{n\neq j}I_n(0)^{a_n+A_n}q_n^{k_n+K_n}\bar{q}_n^{k_n'+K_n'}\right)\\
&&\times\left((k_jK_j'-k_j'K_j)I_j(0)^{a_j+A_j}q_j^{k_j+K_j-1}\bar{q}_j^{k_j'+K_j'-1}\right).
\end{eqnarray*}
Then the coefficient of
\[
\mathcal{M}_{\alpha\kappa\kappa'} :=\prod_{n}I_n(0)^{\alpha_n}q_n^{\kappa_n}\bar{q}_n^{\kappa'_n}
\]
is given by
\begin{equation}\label{006*}
B_{\alpha\kappa\kappa'}= \frac{1}{2\textbf{i}}\sum_{j}\sum_{*}\sum_{**}(k_jK_j'-k_j'K_j)b_{akk'}B_{AKK'},
\end{equation}
where
\begin{equation*}
\sum_{*}=\sum_{a,A \atop a+A=\alpha},
\end{equation*}
and
\begin{equation*}
\sum_{**}=\sum_{\substack{k,k',K,K'\\ \mbox{when}\ n\neq j, k_n+K_n=\kappa_n,k_n'+K_n'=\kappa_n'\\ \mbox{when}\ n=j, k_n+K_n-1=\kappa_n,k_n'+K_n'-1=\kappa_n'}}.
\end{equation*}
To estimate (\ref{H4}), we first note some simple facts:

$\textbf{1}.$ Let
\begin{equation*}
N_1^*=\max\{|n|:A_n+K_n+K_n'\neq0\},
\end{equation*}
and
\begin{equation*}
\nu_1^*=\max\{|n|:\alpha_n+\kappa_n+\kappa_n'\neq0\}.
\end{equation*}
If $j\notin\ \mbox{supp}\ (k+k') \bigcap\ \mbox{supp}\ (K+K')$, then
\begin{equation*}
\frac{\partial\mathcal{M}_{akk'}}{\partial q_j}
\frac{\partial\mathcal{M}_{AKK'}}{\partial \bar{q}_j}-\frac{\partial\mathcal{M}_{akk'}}{\partial \bar{q}_j}
\frac{\partial\mathcal{M}_{AKK'}}{\partial {q}_j}=0.
\end{equation*}
Hence we always assume $j\in\ \mbox{supp}\ (k+k') \bigcap\ \mbox{supp}\ (K+K')$. Therefore one has
\begin{equation}\label{014*}
|j|\leq \min\{n_1^*,N^*_1\}.
\end{equation}
Then the following inequality always holds
\begin{equation}\label{014}
\nu_1^*\leq \max\{n_1^*,N_1^*\},
\end{equation}
Combining (\ref{014*}) and (\ref{014}), one has
\begin{equation*}
\ln^{\sigma}\lfloor j\rfloor+\ln^{\sigma}\lfloor \nu_1^*\rfloor-\ln^{\sigma}\lfloor n_1^*\rfloor-\ln^{\sigma}\lfloor N_1^*\rfloor\leq 0.
\end{equation*}

$\textbf{2}.$ It is easy to see
\begin{eqnarray}
\nonumber\sum_{i\geq 1}\ln^{\sigma}\lfloor n_i^*\rfloor
\nonumber&=&\sum_n(2a_n+k_n+k_n')\ln^{\sigma}\lfloor n\rfloor\\
\nonumber&\geq&\sum_n(2a_n+k_n+k_n')\\
\label{015}&\geq&\sum_n(k_n+k_n')
\end{eqnarray}
and
\begin{eqnarray}
\nonumber\sum_{i\geq 3}\ln^{\sigma}\lfloor N_i^*\rfloor
\nonumber&\geq&\sum_n(2A_n+K_n+K_n')-2\\
\nonumber&\geq&\frac12\sum_n(2A_n+K_n+K_n')\\
\label{016}&\geq&\frac12\sum_{n}(K_n+K_n').
\end{eqnarray}
Based on (\ref{015}) and (\ref{016}), we obtain
\begin{eqnarray}
\sum_{n}(k_n+k_n')(K_n+K_n')\nonumber
&\leq& \left(\sup_{n}(K_n+K_n')\right)\left(\sum_{n}(k_n+k_n')\right)\\
&\leq&\label{017} 2\left(\sum_{i\geq 1}\ln^{\sigma}\lfloor n_i^*\rfloor\right)\left(\sum_{i\geq 3}\ln^{\sigma}\lfloor N_i^*\rfloor\right).
\end{eqnarray}
$\textbf{3}.$ Note that
\begin{equation}\label{012}
\sum_n(2\alpha_n+\kappa_n+\kappa_n')=\sum_n(2a_n+k_n+k_n')+\sum_n(2A_n+K_n+K_n')-2
\end{equation}
and
\begin{eqnarray}
\nonumber&&\sum_n(2\alpha_n+\kappa_n+\kappa_n')\ln^{\sigma}\lfloor n\rfloor\\
 \nonumber&=&\sum_n(2a_n+k_n+k_n')\ln^{\sigma}\lfloor n\rfloor\\
 \label{013}&&+\sum_n(2A_n+K_n+K_n')\ln^{\sigma}\lfloor n\rfloor-2\ln^{\sigma}\lfloor j\rfloor.
\end{eqnarray}
In view of (\ref{016}) and (\ref{012}),\ we have
\begin{eqnarray}\label{009**}
\sum_n(2\alpha_n+\kappa_n+\kappa_n') \leq 2\left(\sum_{i\geq 1}\ln^{\sigma}\lfloor n_i^*\rfloor\right) +2\left(\sum_{i\geq 3}\ln^{\sigma}\lfloor N_i^*\rfloor\right).
\end{eqnarray}
$\textbf{4}.$ It is easy to see
\begin{equation*}
m(\kappa,\kappa')=m(k,k')+m(K,K').
\end{equation*}
Hence,
\begin{equation*}
m^*(\kappa,\kappa')\leq m^*(k,k')+m^*(K,K').
\end{equation*}
Moreover, one has
\begin{equation*}
\ln^{\sigma}\lfloor m^*(\kappa,\kappa'\rfloor\leq \ln^{\sigma}\lfloor m^*(k,k')\rfloor+\ln^{\sigma}\lfloor m^*(K,K')\rfloor.
\end{equation*}
In view of (\ref{H00}) and Lemma \ref{H1}, one has
\begin{eqnarray}
\label{007*} |b_{akk'}|
  &\leq& ||H_1||_{\rho-\delta_{1},\mu+2\delta_{1}}
e^{\rho\sum_{n}(2a_n+k_n+k_n')\ln^{\sigma}\lfloor n\rfloor-2\rho\ln^{\sigma}\lfloor n_1^*\rfloor-\mu \ln^{\sigma}\lfloor m^*(k,k')\rfloor}\nonumber\\
&&
\times
e^{-\frac{1}{2}\delta_{1}
\sum_{i\geq 3}\ln^{\sigma}\lfloor n_i^*\rfloor-\delta_{1} \ln^{\sigma}\lfloor m^*(k,k')\rfloor},
\end{eqnarray}
and
\begin{eqnarray}\label{008*}
|B_{AKK'}|&\leq&||H_2||_{\rho-\delta_{2},\mu+2\delta_{2}}e^{\rho\sum_{n}(2A_n+K_n+K_n')\ln^{\sigma}\lfloor n\rfloor-2\rho\ln^{\sigma}\lfloor N_1^*\rfloor
-\mu \ln^{\sigma}\lfloor m^*(K,K')\rfloor}\nonumber\\
&&
\times
e^{-\frac{1}{2}\delta_{2}\sum_{i\geq 3}\ln^{\sigma}\lfloor N_i^*\rfloor-\delta_{2} \ln^{\sigma}\lfloor m^*(K,K')\rfloor}.
\end{eqnarray}
Substitution of (\ref{007*}) and (\ref{008*}) in (\ref{006*}) gives
\begin{eqnarray*}
|B_{\alpha\kappa\kappa'}|&\leq&\frac{1}{2} ||H_1||_{\rho-\delta_{1},\mu+2\delta_{1}}||H_2||_{\rho-\delta_{2},\mu+2\delta_{2}}
\sum_{j}\sum_{*}\sum_{**}\left\{|k_jK_j'-k_j'K_j|\right.\\
\nonumber &&\times e^{\rho\left(\underset{n}{\sum}(2(a_n+A_{n})+k_n +K_{n}+k_n'+K_{n}')\ln^{\sigma}\lfloor n\rfloor-2\ln^{\sigma}\lfloor n_1^*\rfloor-2\ln^{\sigma}\lfloor N_1^*\rfloor\right)}\\
\nonumber &&\times e^{-\mu (\ln^{\sigma}\lfloor m^{\ast}(k,k')\rfloor+ \ln^{\sigma}\lfloor m^*(K,K')\rfloor)}\\
\nonumber &&\left.\times e^{-\frac{1}{2}\left(\delta_{1}\sum_{i\geq 3}\ln^{\sigma}\lfloor n_i^*\rfloor+\delta_{2}\sum_{i\geq 3}\ln^{\sigma}\lfloor N_i^*\rfloor\right)}\right.\\
\nonumber && \left.\times e^{-\delta_{1} \ln^{\sigma}\lfloor m^*(k,k')\rfloor-\delta_{2} \ln^{\sigma}\lfloor M^*(K,K')\rfloor}\right\}.
\end{eqnarray*}
Then one has
\begin{eqnarray*}
\nonumber \left|\left|\{H_1,H_2\}\right|\right|_{\rho,\mu}
 &\leq& \frac{1}{2}||H_1||_{\rho-\delta_{1},\mu+2\delta_{1}}||H_2||_{\rho-\delta_{2},\mu+2\delta_{2}}\sum_{j}\sum_{*}\sum_{**}\left\{ |k_jK_j'-k_j'K_j|\right.\\
\nonumber &&\left.  \times e^{2\rho(\ln^{\sigma}\lfloor j\rfloor+\ln^{\sigma}\lfloor \nu_1^*\rfloor-\ln^{\sigma}\lfloor n_1^*\rfloor-\ln^{\sigma}\lfloor N_1^*\rfloor)}\right.\\
\nonumber &&\left.\times e^{-\frac{1}{2}\left(\delta_{1}\sum_{i\geq 3}\ln^{\sigma}\lfloor n_i^*\rfloor)+\delta_{2}\sum_{i\geq 3}\ln^{\sigma}\lfloor N_i^*\rfloor\right)}\right.\\
\nonumber && \left.\times e^{-\delta_{1} \ln^{\sigma}\lfloor m^*(k,k')\rfloor-\delta_{2} \ln^{\sigma}\lfloor M^*(K,K')\rfloor}\right\}.
\end{eqnarray*}
To show (\ref{H4}) holds, it suffices to prove
\begin{equation}\label{H5}
I\leq \frac{1}{\delta_{2}}\exp\left\{\frac{300}{\delta_{1}}\exp\left\{\left(\frac{50}{\delta_{1}}
\right)^{\frac{1}{\sigma-1}}\right\}\right\},
\end{equation}
where
\begin{eqnarray*}
I&=& \frac{1}{2}\sum_{j}\sum_{*}\sum_{**}\left\{|k_jK_j'-k_j'K_j|e^{2\rho(\ln^{\sigma}\lfloor j\rfloor+\ln^{\sigma}\lfloor \nu_1^*\rfloor-\ln^{\sigma}\lfloor n_1^*\rfloor-\ln^{\sigma}\lfloor N_1^*\rfloor)}\right.\\
&&\left.\times e^{-\frac{1}{2}\left(\delta_{1}\sum_{i\geq 3}\ln^{\sigma}\lfloor n_i^*\rfloor)+\delta_{2}\sum_{i\geq 3}\ln^{\sigma}\lfloor N_i^*\rfloor\right)}\right.\\
\nonumber && \left.\times e^{-\delta_{1} \ln^{\sigma}\lfloor m^*(k,k')\rfloor-\delta_{2} \ln^{\sigma}\lfloor M^*(K,K')\rfloor}\right\}.
\end{eqnarray*}

Now we will prove the inequality (\ref{H5}) holds:

$\textbf{Case. 1.} \ \nu_1^*\leq N_1^*$.

$\textbf{Subase. 1.1.}\ |j|\leq n_3^*\Rightarrow \lfloor j\rfloor\leq \lfloor n_3^*\rfloor.$

Using $ 0< \delta_{1}< \frac{\rho}{4}$, one has
\begin{eqnarray}\label{019}
e^{2\rho(\ln^{\sigma}\lfloor j\rfloor- \ln^{\sigma}\lfloor n_1^*\rfloor)}
&\leq& e^{\frac{1}{2}\delta_{1}(\ln^{\sigma}\lfloor n_3^*\rfloor-\ln^{\sigma}\lfloor n_1^*\rfloor)},
\end{eqnarray}
Hence one obtains
\begin{eqnarray}
\nonumber &&e^{2\rho(\ln^{\sigma}\lfloor j\rfloor+\ln^{\sigma}\lfloor \nu_1^*\rfloor-\ln^{\sigma}\lfloor n_1^*\rfloor-\ln^{\sigma}\lfloor N_1^*\rfloor)}
e^{-\frac{1}{2}\delta_{1}\sum_{i\geq 3}\ln^{\sigma}\lfloor n_i^*\rfloor}\\
&\leq&\nonumber e^{-\frac{1}{2}\delta_{1}(\ln^{\sigma}\lfloor n_1^*\rfloor+\sum_{i\geq 4}\ln^{\sigma}\lfloor n_i^*\rfloor)}\\
\label{021}&\leq& e^{-\frac{1}{6}\delta_{1}\sum_{i\geq 1}\ln^{\sigma}\lfloor n_i^*\rfloor}.
\end{eqnarray}
\begin{rem}\label{022}
Note that if $j,a,k,k'$ are specified, and then $A,K,K'$ are uniquely determined.
\end{rem}
In view of (\ref{021}), we have
\begin{eqnarray*}
I&\leq&\frac{1}{2}\sum_{j}\sum_{*}\sum_{**}\left\{(k_j+k_j')(K_j+K_j')
e^{-\frac{1}{6}\delta_{1}\sum_{i\geq 1}n_i^*\rfloor}e^{-\frac{1}{2}\delta_{2}\sum_{i\geq 3}\ln^{\sigma}\lfloor N_i^*\rfloor}\right.\\
&\leq&\sum_{a,k,k'}\left(\sum_{i\geq 1}\ln^{\sigma}\lfloor n_i^*\rfloor\right)\left(\sum_{i\geq 3}\ln^{\sigma}\lfloor N_i^*\rfloor\right)e^{-\frac{1}{6}\delta_{1}\sum_{i\geq 1}\ln^{\sigma}\lfloor n_i^*\rfloor}e^{-\frac{1}{2}\delta_{2}\sum_{i\geq 3}\ln^{\sigma}\lfloor N_i^*\rfloor}\\
&&\mbox{(in view of the inequality (\ref{017}))}\\
\nonumber &\leq&\frac{24}{e^{2}\delta_{1}\delta_{2}}\sum_{a,k,k'}e^{-\frac{1}{12}\delta_{1}\sum_{i\geq 1}\ln^{\sigma}\lfloor n_i^*\rfloor}\\
&\leq&\frac{24}{e^{2}\delta_{1}\delta_{2}}\prod_{n\in\mathbb{Z}}
\left({1-e^{-\frac{1}{12}\delta_1\ln^{\sigma}\lfloor n\rfloor}}\right)^{-1}
\left({1-e^{-\frac{1}{12}\delta_{1}\ln^{\sigma}\lfloor n\rfloor}}\right)^{-2}\\
&&\nonumber\mbox{(which is based on (\ref{041809}))}\\
&\leq& \frac{1}{\delta_{2}}\exp\left\{\frac{300}{\delta_{1}}\exp\left\{\left(\frac{50}{\delta_{1}}
\right)^{\frac{1}{\sigma-1}}\right\}\right\},
\end{eqnarray*}
where the last inequality is based on (\ref{122401}).

$\textbf{Subcase. 1.2.}\ j\in\{n_1,n_2\},\ |n_1|=n_1^*,\ |n_2|=n_2^*.$

If $2a_j+k_j+k'_j>2$, then $|j|\leq n_3^*$, we are in $\textbf{Subcase. 1.1.}$. Hence in what follows, we always assume
\begin{equation*}
2a_j+k_j+k'_j\leq2,
\end{equation*}
which implies
\begin{equation}\label{023}
k_j+k_j'\leq 2
\end{equation}
and
\begin{equation}\label{024}
n_2^*>n_3^*.
\end{equation}
From (\ref{023}) and in view of $j\in\{n_1,n_2\}$, it follows that
\begin{eqnarray*}
I\nonumber
\nonumber&\leq&\sum_{a,k,k'}\left\{(K_{n_1}+K'_{n_1}+K_{n_2}+K'_{n_2})\right. \\
\label{025*}&&\left.\times e^{-\frac{1}{2}\left(\delta_{1}\sum_{i\geq 3}\ln^{\sigma}\lfloor n_i^*\rfloor+\delta_{2}\sum_{i\geq 3}\ln^{\sigma}\lfloor N_i^*\right)-\delta_{1}\ln^{\sigma}\lfloor m^*(k,k')\rfloor}\right\}.
\end{eqnarray*}
Since
\begin{equation*}
 K_j+K'_j\leq \kappa_j+\kappa'_j-k_j-k'_j+2\leq\kappa_j+\kappa'_j+2, \forall j,
\end{equation*}
one has
\begin{eqnarray}
\nonumber I\nonumber&\leq& \sum_{a,k,k'}\left\{(\kappa_{n_1}+\kappa'_{n_1}+\kappa_{n_2}+\kappa'_{n_2}+4)\right. \\
\nonumber&&\left.\times e^{-\frac{1}{2}\left(\delta_{1}\sum_{i\geq 3}\ln^{\sigma}\lfloor n_i^*\rfloor+\delta_{2}\sum_{i\geq 3}\ln^{\sigma}\lfloor N_i^*\rfloor\right)-\delta_{1}\ln^{\sigma}\lfloor m^*(k,k')\rfloor}\right\}\\
 &\leq& \nonumber \sum_{a,k,k'}\left\{(\kappa_{n_1}+\kappa'_{n_1}+\kappa_{n_2}+\kappa'_{n_2}+4)\right. \\
\nonumber&&\times e^{-\frac14\delta_{1}\sum_{i\geq 3}\ln^{\sigma}\lfloor n_i^*\rfloor-\delta_{1} \ln^{\sigma}\lfloor m^*(k,k')\rfloor}\  \mbox{ ( based  on  (\ref{009**}) )}\\
&&\nonumber \left.\times e^{-\frac18\delta\sum_{n}(2\alpha_n+\kappa_n+\kappa'_n)}\right\}\\
\label{025} &=& \sum_{l\in\mathbb{Z}}\sum_{a,k,k',\atop{m(k,k')=l}}\left\{(\kappa_{n_1}+\kappa'_{n_1}+\kappa_{n_2}+\kappa'_{n_2}+4)\right. \\
\nonumber&&\times e^{-\frac14\delta_{1}\sum_{i\geq 3}\ln^{\sigma}\lfloor n_i^*\rfloor-\delta_{1}\ln^{\sigma}\lfloor l\rfloor }\\
&&\nonumber \left.\times e^{-\frac18\delta\sum_{n}(2\alpha_n+\kappa_n+\kappa'_n)}\right\},
\end{eqnarray}
where $\delta_{1}\wedge\delta_{2}=\min\{\delta_{1},\delta_{2}\} $.
\begin{rem}\label{026} Obviously, $\{n_1,n_2\}\bigcap \mathrm{supp}\ \mathcal{M}_{\alpha\kappa\kappa'}\neq \emptyset$, and if $n_1$ (resp. $n_2$), $\{n_i\}_{i\geq 3}$ and $m(k,k')=l$  is specified, then $n_2$ (resp. $n_1$) is determined uniquely. Thus $n_1,n_2$ range in a set of cardinality no more than \begin{equation}\label{027}\#\mathrm{supp} \ \mathcal{M}_{\alpha\kappa\kappa'}\leq\sum_{n}(2\alpha_n+\kappa_n+\kappa'_n).
\end{equation}
\end{rem}
Also, if $\{n_i\}_{i\geq 1}$ is given, then $\{2a_n+k_n+k'_n\}_{n\in\mathbb{Z}}$ is specified, and hence $(a,k,k')$ is specified up to a factor of
$$\prod_{n}(1+l_n^2),$$
where
$$l_n=\#\{j:n_j=n\}.$$
Following the inequality (\ref{025}), we thus obtain
\begin{eqnarray}
\nonumber I&\leq& \sum_{l\in\mathbb{Z}}\sum_{\{n_i\}_{i\geq1}}\left\{\prod_{m}(1+l_m^2)(\kappa_{n_1}+\kappa'_{n_1}+\kappa_{n_2}+\kappa'_{n_2}+4) \right.\\
\nonumber&&\left.\times e^{-\frac14\delta_{1}\sum_{i\geq 3}\ln^{\sigma}\lfloor n_i^*\rfloor-\delta_{1} \ln^{\sigma}\lfloor l\rfloor}\cdot e^{-\frac18\delta\sum_{n}(2\alpha_n+\kappa_n+\kappa'_n)}\right\}\\
\nonumber&\leq& 5\sum_{l\in\mathbb{Z}}\sum_{\{n_i\}_{i\geq3}}\left\{\prod_{|m|\leq n_3^*}(1+l_m^2)\left(\sum_{n_1,n_2}(\kappa_{n_1}+\kappa'_{n_1}+\kappa_{n_2}+\kappa'_{n_2}+4)\right)\right. \\
\nonumber&& \times e^{-\frac14\delta_{1}\sum_{i\geq 3}\ln^{\sigma}\lfloor n_i^*\rfloor
-\delta_{1} \ln^{\sigma}\lfloor l\rfloor}\left.\cdot e^{-\frac18\delta\sum_{n}(2\alpha_n+\kappa_n+\kappa'_n)}\right\}\\
&&\nonumber\mbox{(in view of $\prod_{|m|>n^*_1}(1+l_m^2)=1$ and $\prod_{m\in\{n_1,n_2\}}(1+l_m^2)\leq 5$)}\\
\nonumber& \leq & 5\sum_{l\in\mathbb{Z}}\sum_{\{n_i\}_{i\geq3}}\left\{\prod_{|m|\leq n_3^*}(1+l_m^2)\left(\sum_{n}(\kappa_{n}+\kappa'_{n})+4\#\mathrm{supp}\ \mathcal{M}_{\alpha\kappa\kappa'}\right)\right. \\
\nonumber&& \times 5e^{-\frac14\delta_{1}\sum_{i\geq 3}\ln^{\sigma}\lfloor n_i^*\rfloor-\delta_{1} \ln^{\sigma}\lfloor l\rfloor}\left.\cdot e^{-\frac18\delta\sum_{n}(2\alpha_n+\kappa_n+\kappa'_n)}\right\}\\
&&\nonumber\mbox{(the inequality is based on Remark \ref{026})}\\
\nonumber&\leq& \frac{200}{e\delta}\left(\sum_{l\in\mathbb{Z}}e^{-\delta_{1} \ln^{\sigma}\lfloor l\rfloor}\right)\left(\sum_{\{n_i\}_{i\geq3}}\prod_{|m|\leq n_3^*}(1+l_m^2)e^{-\frac14\delta_{1}\sum_{i\geq 3}\ln^{\sigma}\lfloor n_i^*\rfloor}\right)\\
&&\mbox{ (based on (\ref{027}) and (\ref{042805*}))}\nonumber\\
\nonumber&\leq&  \frac{600}{e\delta\delta_{1}}\exp\left\{\left(\frac{2}{\delta_{1}\sigma}
\right)^{\frac{1}{\sigma-1}}\right\}\left(\sum_{\{l_m\}_{|m|\leq n_3^{\ast}}}
e^{-\frac16\delta_{1}\sum_{|m|\leq n_3^\ast}l_m\ln^{\sigma}\lfloor m\rfloor}\right)\\
\nonumber&&\times\sup_{\{l_m\}_{|m|\leq n_3^\ast}}\left(\prod_{|m|\leq n_3^\ast}(1+l_m^2)
e^{-\frac{1}{12}\delta_{1}\sum_{|m|\leq n_3^\ast}l_m \ln^{\sigma}\lfloor m\rfloor^{\theta}}\right)\\
&&\mbox{ (based on (\ref{0418011}))}\nonumber\\
\nonumber&\leq& \frac{600}{e\delta\delta_{1}}\exp\left\{\left(\frac{2}{\delta_{1}\sigma}
\right)^{\frac{1}{\sigma-1}}\right\}
\exp\left\{6\left(\frac{48}{\delta_{1}}\right)^{\frac{1}{\sigma-1}}
\exp\left\{\left(\frac{24}{\delta_{1}}
\right)^{\frac{1}{\sigma}}\right\}\right\}\\
\nonumber&&\times
\left(\prod_{m\in\mathbb{Z}}\frac{1}{1-e^{-\frac16\delta_{1} \ln^{\sigma}\lfloor m\rfloor}}\right)\quad \mbox{ (in view of (\ref{042807}))}\nonumber\\
\nonumber&\leq&\frac{1}{\delta_{2}}\exp\left\{\frac{300}{\delta_{1}}\exp\left\{\left(\frac{50}{\delta_{1}}
\right)^{\frac{1}{\sigma-1}}\right\}\right\},
\end{eqnarray}
where the last inequality is based on (\ref{122401}).

$\textbf{Case. 2.}\ \nu_1^*>N_1^*.$

In view of (\ref{014}), one has $n_1^*=\nu_1^*$. Hence,  $n_2$ is determined by $n_1$, $\{n_i\}_{i\geq 3}$ and the momentum $m(k,k')$. Similar to Case 1.2, we have
\begin{eqnarray*}
I&\leq&\frac{1}{\delta_{2}}\exp\left\{\frac{300}{\delta_{1}}\exp\left\{\left(\frac{50}{\delta_{1}}
\right)^{\frac{1}{\sigma-1}}\right\}\right\}.
\end{eqnarray*}

Therefore,\ we finish the proof of (\ref{H4}).
\end{proof}

\begin{lem}\label{E1}
Let $\sigma>2, \rho>0$ and
\[
0<\delta_1,\delta_2< \min\{\frac{1}{4}\rho, 3-2\sqrt{2}\}.
\]
Assume further \begin{equation}\label{035}
\frac{e}{\delta_{2}}\exp\left\{\frac{300}{\delta_{1}}\exp\left\{\left(\frac{50}{\delta_{1}}
\right)^{\frac{1}{\sigma-1}}\right\}\right\}||F||_{\rho-\delta_1,\mu+2\delta_{1}} \ll 1.
\end{equation}
Then for any Hamiltonian function $H$, we get
\begin{equation}\nonumber
||H\circ\Phi_F||_{\rho,\mu}
\leq\left(1+\frac{e}{\delta_{2}}\exp\left\{\frac{300}{\delta_{1}}\exp\left\{\left(\frac{50}{\delta_{1}}
\right)^{\frac{1}{\sigma-1}}\right\}\right\}
||F||_{\rho-\delta_1,\mu+2\delta_{1}}\right)
||H||_{\rho-\delta_2,\mu+2\delta_{2}}.
\end{equation}
\end{lem}

\begin{proof}
Firstly,  we expand $H\circ\Phi_F$ into the Taylor series
 \begin{equation}\label{037}
 H\circ\Phi_F=\sum_{n\geq 0}\frac{1}{n!}H^{(n)},
 \end{equation}
where $H^{(n)}=\{H^{(n-1)},F\}$ and $H^{(0)}=H$.

We will estimate $||H^{(n)}||_{\rho,\mu}$ by using Lemma \ref{H3} again and again:
\begin{eqnarray}
\nonumber&&||H^{(n)}||_{\rho,\mu}=||\{H^{(n-1)},F\}||_{\rho,\mu}\\
\nonumber&\leq&\left(\exp\left\{\frac{300}{\delta_{1}}\exp\left\{\left(\frac{50}{\delta_{1}}
\right)^{\frac{1}{\sigma-1}}\right\}\right\}
||F||_{\rho-\delta_1,\mu+2\delta_{1}}\right)^n\left(\frac{n}{\delta_2}\right)^n
||H||_{\rho-\delta_2,\mu+2\delta_{2}}.
\end{eqnarray}
Hence in view of (\ref{037}), one has
\begin{eqnarray*}
&&||H\circ\Phi_F||_{\rho,\mu}\\
\nonumber&\leq&\sum_{n\geq 0}\frac{n^n}{n!}\left(\frac{1}{\delta_{2}}\exp\left\{\frac{300}{\delta_{1}}\exp\left\{\left(\frac{50}{\delta_{1}}
\right)^{\frac{1}{\sigma-1}}\right\}\right\}
||F||_{\rho-\delta_1,\mu+2\delta_{1}}\right)^{n}||H||_{\rho-\delta_2,\mu+2\delta_{2}}\\
&\leq&\sum_{n\geq 0}
\left(\frac{e}{\delta_{2}}\exp\left\{\frac{300}{\delta_{1}}\exp\left\{\left(\frac{50}{\delta_{1}}
\right)^{\frac{1}{\sigma-1}}\right\}\right\}
||F||_{\rho-\delta_1,\mu+2\delta_{1}}\right)^{n}||H||_{\rho-\delta_2,\mu+2\delta_{2}}\\
&&(\mbox{in view of $n^n<n!e^n$})\\
&\leq & \left(1+\frac{e}{\delta_{2}}\exp\left\{\frac{300}{\delta_{1}}\exp\left\{\left(\frac{50}{\delta_{1}}
\right)^{\frac{1}{\sigma-1}}\right\}\right\}
||F||_{\rho-\delta_1,\mu+2\delta_{1}}\right)
||H||_{\rho-\delta_2,\mu+2\delta_{2}}.
\end{eqnarray*}
\end{proof}

Finally, we give the estimate of the Hamiltonian vector field.
\begin{lem}\label{H6}
Given a Hamiltonian
\begin{equation}\label{028}
H=\sum_{a,k,k'\in\mathbb{N}^{\mathbb{Z}}}B_{akk'}\mathcal{M}_{akk'},
\end{equation}
then for any $\mu>r> 5\rho, ||q||_{r,\infty} < 1$ and $ ||I(0)||_{r,\infty} < 1 $,\ one has
\begin{equation}\label{029}
\sup_{j \in \mathbb{Z}}\left|e^{r \ln^{\sigma}\lfloor j\rfloor}\frac{\partial{H}}{\partial q_{j}}\right|\leq C(r,\rho,\mu,\sigma)||H||_{\rho,\mu},
\end{equation}
where $C(r,\rho,\mu,\sigma)$ is a positive constant depending on $r,\rho,\mu$ and $\sigma$ only,\ and
\begin{equation}\label{030}
||I(0)||_{r,\infty} := \sup_{n\in \mathbb{Z}}|I_{n}(0)|e^{2r \ln^{\sigma}\lfloor n\rfloor}.
\end{equation}
\end{lem}

\begin{proof}
In view of (\ref{028}) and for each $ j \in \mathbb{Z}$,\ one has
\begin{equation*}
\frac{\partial{H}}{\partial q_{j}}=\sum_{a,k,k'}B_{akk'}\left(\prod_{n\neq j}I_{n}(0)^{a_{n}}q_{n}^{k_{n}}
\bar{q}_{n}^{k_{n}'}\right)\left(k_{j}
I_{j}(0)^{a_{j}}q_{j}^{k_{j}-1}\bar{q}_{j}^{k_{j}'}\right).
\end{equation*}
Now we would like to estimate
\begin{equation}\label{031}
\left|e^{r \ln^{\sigma}\lfloor j\rfloor}\frac{\partial{H}}{\partial q_{j}}\right|=\left|e^{r \ln^{\sigma}\lfloor j\rfloor}\sum_{a,k,k'}B_{akk'}\left(\prod_{n\neq j}I_{n}(0)^{a_{n}}q_{n}^{k_{n}}
\bar{q}_{n}^{k_{n}'}\right)\left(k_{j}
I_{j}(0)^{a_{j}}q_{j}^{k_{j}-1}
\bar{q}_{j}^{k_{j}'}\right)\right|.
\end{equation}
Based on (\ref{H00}), one has
\begin{equation}\label{032}
|B_{akk'}|\leq ||H||_{\rho,\mu}e^{\rho(\sum_{n}(2a_{n}+k_{n}+k_{n}') \ln^{\sigma}\lfloor n\rfloor
-2 \ln^{\sigma}\lfloor n_1^\ast\rfloor)-\mu  \ln^{\sigma}\lfloor m^\ast(k,k')\rfloor}.
\end{equation}
In view of $||q||_{r,\infty}<1$ and $ ||I(0)||_{r,\infty} < 1$, one has
\begin{eqnarray}\label{033}
|q_{n}| < e^{-r \ln^{\sigma}\lfloor n\rfloor},
\end{eqnarray}
and
\begin{eqnarray}\label{034}
|I_{n}(0)| < e^{-2r \ln^{\sigma}\lfloor n\rfloor}.
\end{eqnarray}
Substituting (\ref{033}) and (\ref{034}) into (\ref{032}),\ one has
\begin{eqnarray}
|(\ref{032})|
&\leq&\nonumber||H||_{\rho,\mu}\left|e^{r \ln^{\sigma}\lfloor j\rfloor}\sum_{l\in\mathbb{Z}}\sum_{a,k,k',\atop{m(k,k')=l}}\left\{
k_je^{\rho(\underset{n}{\sum}(2a_n+k_n+k_n') \ln^{\sigma}\lfloor n\rfloor-2 \ln^{\sigma}\lfloor n_1^*\rfloor)}\right.\right.\\ \nonumber &&\left.\left.\cdot
e^{-r(\underset{n}{\sum}(2a_n+k_n+k_n') \ln^{\sigma}\lfloor n\rfloor- \ln^{\sigma}\lfloor j\rfloor)-\mu \ln^{\sigma}\lfloor l\rfloor}\right\}\right|\\
&=&\nonumber||H||_{\rho,\mu}\left|\sum_{l\in\mathbb{Z}}\sum_{a,k,k',\atop{m(k,k')=l}}\left\{
k_je^{\rho(\underset{n}{\sum}(2a_n+k_n+k_n') \ln^{\sigma}\lfloor n\rfloor-2 \ln^{\sigma}\lfloor n_1^*\rfloor)}\right.\right.\\ \nonumber &&\left. \left.\cdot
e^{-r(\underset{n}{\sum}(2a_n+k_n+k_n') \ln^{\sigma}\lfloor n\rfloor-2 \ln^{\sigma}\lfloor j\rfloor)-\mu \ln^{\sigma}\lfloor l\rfloor}\right\}\right|.
\end{eqnarray}
Then we only need to estimate
\begin{eqnarray}\label{L1}
&&\left|\sum_{l\in\mathbb{Z}}\sum_{a,k,k',\atop{m(k,k')=l}}\left\{
k_je^{\rho(\underset{n}{\sum}(2a_n+k_n+k_n') \ln^{\sigma}\lfloor n\rfloor-2 \ln^{\sigma}\lfloor n_1^*\rfloor)}\right.\right.\\
\nonumber &&\left. \left.\cdot
e^{-r(\underset{n}{\sum}(2a_n+k_n+k_n') \ln^{\sigma}\lfloor n\rfloor-2 \ln^{\sigma}\lfloor j\rfloor)-\mu \ln^{\sigma}\lfloor l\rfloor}\right\}\right|.
\end{eqnarray}

Now we will estimate the last inequality in the following two cases:

\textbf{Case 1.} $|j|\leq n_3^*$.

Then one has
\begin{eqnarray*}
(\ref{L1})
&\leq& \left|\sum_{a,k,k'}k_je^{\rho\sum_{i\geq 1}\ln^{\sigma}\lfloor n_i^*\rfloor}e^{-r(n_1^*)^{\theta}-r\sum_{i\geq 4}\ln^{\sigma}\lfloor n_i^*\rfloor}\right|\left( \sum_{l\in \mathbb{Z}}e^{-\mu\ln^{\sigma}\lfloor l\rfloor}\right)\\
\nonumber&\leq&\sum_{a,k,k'}\left(\sum_{i\geq 1}\ln^{\sigma}\lfloor n_i^*\rfloor\right)e^{\frac13(-r+3\rho)\sum_{i\geq 1}\ln^{\sigma}\lfloor n_i^*\rfloor}\left( \sum_{l\in \mathbb{Z}}e^{-\mu\ln^{\sigma}\lfloor l\rfloor}\right)\\
\nonumber&\leq& \frac{3}{\mu}\exp\left\{\left(\frac{2}{\mu\sigma}\right)^{\frac{1}{\sigma-1}}\right\}\left(\frac{12}{e(r-3\rho)}\right)\left(\sum_{a,k,k'}
e^{\frac14(-r+3\rho)\sum_{i\geq 1}\ln^{\sigma}\lfloor n_i^*\rfloor}\right)\\
&&\nonumber \quad\mbox{(in view of (\ref{0418011}))}\\
\nonumber&=&  \frac{3}{\mu}\exp\left\{\left(\frac{2}{\mu\sigma}\right)^{\frac{1}{\sigma-1}}\right\}\left(\frac{12}{e(r-3\rho)}\right) \sum_{a,k,k'}
e^{\frac14(-r+3\rho)\underset{n}{\sum}(2a_n+k_n+k'_n)\ln^{\sigma}\lfloor n\rfloor}\\
\nonumber&\leq&\frac{3}{\mu}\exp\left\{\left(\frac{2}{\mu\sigma}\right)^{\frac{1}{\sigma-1}}\right\}\left(\frac{12}{e(r-3\rho)}\right)\prod_{n\in\mathbb{Z}}
\left({1-e^{\frac12(-r+3\rho)\ln^{\sigma}\lfloor n\rfloor}}\right)^{-1}\\
\nonumber&&\times\prod_{n\in\mathbb{Z}}
\left({1-e^{\frac14(-r+3\rho)\ln^{\sigma}\lfloor n\rfloor}}\right)^{-2} \quad
\mbox{(in view of (\ref{041809}))}\\
\nonumber&\leq& \exp\left\{\left(\frac{2}{\mu}\right)^{\frac{1}{\sigma-1}}\right\}\left(\frac{3}{\rho}\right)\prod_{n\in\mathbb{Z}}
\left({1-e^{-\rho\ln^{\sigma}\lfloor n\rfloor}}\right)^{-1}
\nonumber\prod_{n\in\mathbb{Z}}
\left({1-e^{-\frac{1}{2}\rho\ln^{\sigma}\lfloor n\rfloor}}\right)^{-2}\\
&&\nonumber  \ {\mbox{(in view of $r>5\rho$)}}    \\
&\leq&\nonumber \exp\left\{\left(\frac{2}{\mu}\right)^{\frac{1}{\sigma-1}}\right\}
\exp\left\{\frac{100}{\rho}\cdot\exp\left\{\left(\frac{8}{\rho}\right)^{\frac{1}{\sigma-1}}\right\}\right\},
\end{eqnarray*}
where the last inequality is based on (\ref{122401}).

\textbf{Case 2.} $|j|> n_3^*$, which implies $k_j\leq 2$.

Then one has
\begin{eqnarray}
\nonumber(\ref{L1})
&\leq&2 \left|\sum_{l\in\mathbb{Z}}\sum_{a,k,k',\atop{m(k,k')=l}}e^{\rho\sum_{i\geq 3}\ln^{\sigma}\lfloor n_i^*\rfloor}e^{-\frac{1}{2}r\sum_{i\geq 3}\ln^{\sigma}\lfloor n_i^*\rfloor}e^{-(\mu-r)\ln^{\sigma}\lfloor l\rfloor}\right|\\
\nonumber&=& 2\left|\sum_{l\in\mathbb{Z}}\sum_{a,k,k',\atop{m(k,k')=l}}e^{(-\frac{1}{2}r+\rho)\sum_{i\geq 3}\ln^{\sigma}\lfloor n_i^*\rfloor}e^{-(\mu-r)\ln^{\sigma}\lfloor l\rfloor}\right|:= A.
\end{eqnarray}
If $\{n_i\}_{i\geq 1}$ is given, then $\{2a_n+k_n+k'_n\}_{n\in\mathbb{Z}}$ is specified, and hence $(a,k,k')$ is specified up to a factor of
$\prod_{n}(1+l_n^2),$
where
$l_n=\#\{j:n_j=n\}.$
Since $|j|>n_3^*$, then $j\in\{n_1,n_2\}$. Hence, if $(n_i)_{i\geq 3}$ and $j,m^{\ast}(k,k')$ are given, then $n_1$ and $n_2$ are uniquely determined. Then, one has
\begin{eqnarray}
\nonumber A &\leq&  2\left|\sum_{l\in\mathbb{Z}}\sum_{(n_i)_{i\geq3}}\prod_{|n|\leq n_1^*}(1+l_n^2)e^{(-(2-2^{\theta})r+\rho)\sum_{i\geq 3}\ln^{\sigma}\lfloor n_i^*\rfloor}e^{-(\mu-r)\ln^{\sigma}\lfloor l\rfloor}\right| \\
\nonumber&\leq&10\left|\sum_{l\in\mathbb{Z}}\sum_{(n_i)_{i\geq3}}\prod_{|n|\leq n_3^*}(1+l_n^2)e^{(-\frac{1}{2}r+\rho)\sum_{i\geq 3}\ln^{\sigma}\lfloor n_i^*\rfloor}e^{-(\mu-r)\ln^{\sigma}\lfloor l\rfloor}\right| \\
\nonumber && \mbox{ ( in view of $ \prod_{n \in \{n_1, n_{2}\}}(1+l_n^2) \leq 5 )$ }\\
\nonumber&\leq&10\left(\sum_{(n^*_i)_{i\geq3}}e^{-\rho\sum_{i\geq 3}\ln^{\sigma}\lfloor n_i^*\rfloor}\right)\sup_{(n^\ast_i)_{i\geq3}}\left(\prod_{|n|\leq n_3^*}(1+l_n^2)e^{-\frac{1}{2}\rho\sum_{i\geq 3}\ln^{\sigma}\lfloor n_i^*\rfloor}\right)\\
\nonumber && \times\left(\sum_{l\in\mathbb{Z}}e^{-(\mu-r)\ln^{\sigma}\lfloor l\rfloor}\right)\  \mbox{ ( in view of $ r > 5\rho $ ) }\\
\nonumber&\leq&10\left(\frac{3}{\mu-r}\right)
\exp\left\{\left(\frac{2}{(\mu-r)\sigma}\right)^{\frac{1}{\sigma-1}}\right\}
\exp\left\{6\left(\frac{8}{\rho}\right)^{\frac{1}{\sigma-1}}\cdot
\exp\left\{\left(\frac{4}{\rho}\right)^{\frac{1}{\sigma}}\right\}\right\}\\
\nonumber&&\times
\left(\sum_{(n^*_i)_{i\geq3}}e^{-\rho\sum_{i\geq 3}(n_i^*)^{\theta}}\right) \quad
\mbox{( in view of (\ref{0418011}) and (\ref{042807}))}\\
\nonumber&=& 10\left(\frac{3}{\mu-r}\right)
\exp\left\{\left(\frac{2}{(\mu-r)\sigma}\right)^{\frac{1}{\sigma-1}}\right\}
\exp\left\{6\left(\frac{8}{\rho}\right)^{\frac{1}{\sigma-1}}\cdot
\exp\left\{\left(\frac{4}{\rho}\right)^{\frac{1}{\sigma}}\right\}\right\}\\
\nonumber&&\times\left(\sum_{(l_n)_{|n|\leq n_3^*}}e^{-\rho\sum_{|n|\leq n_3^*}l_n\ln^{\sigma}\lfloor n\rfloor}\right)\\
\nonumber&\leq& \exp\left\{\left(\frac{2}{\mu-r}\right)^{\frac{1}{\sigma-1}}\right\}
\exp\left\{\frac{20}{\rho}\cdot
\exp\left\{\left(\frac{4}{\rho}\right)^{\frac{1}{\sigma-1}}\right\}\right\},
\end{eqnarray}
where the last inequality is based on (\ref{041809}).\\
Hence, we finished the proof of (\ref{029}).
\end{proof}

\section{KAM iteration}
\subsection{Derivation of homolnical equations}
The proof of Theorem \ref{Thm} employs the rapidly
converging iteration scheme of Newton type to deal with small divisor problems
introduced by Kolmogorov, involving the infinite sequence of coordinate transformations.
At the $s$-th step of the scheme, a Hamiltonian
$H_{s} = N_{s} + R_{s}$
is considered as a small perturbation of some normal form $N_{s}$. A transformation $\Phi_{s}$ is
set up so that
$$ H_{s}\circ \Phi_{s} = N_{s+1} + R_{s+1}$$
with another normal form $N_{s+1}$ and a much smaller perturbation $R_{s+1}$. We drop the index $s$ of $H_{s}, N_{s}, R_{s}, \Phi_{s}$ and shorten the index $s+1$ as $+$.

Rewrite $R$ as
\begin{equation}\label{N1}
R=R_0+R_1+R_2
\end{equation}
where
\begin{eqnarray*}
{R}_0&=&\sum_{a,k,k'\in{\mathbb{N}^{\mathbb{Z}}}\atop\mbox{supp}\ k\bigcap \mbox{supp}\ k'=\emptyset}B_{akk'}\mathcal{M}_{akk'},\\
{R}_1&=&\sum_{m\in\mathbb{Z}}J_m\left(\sum_{a,k,k'\in{\mathbb{N}^{\mathbb{Z}}}\atop\mbox{supp}\ k\bigcap \mbox{supp}\ k'=\emptyset}B_{akk'}^{(m)}\mathcal{M}_{akk'}\right),\\
{R}_2&=&\sum_{m_1,m_2\in\mathbb{Z}}J_{m_1}J_{m_2}\left(\sum_{a,k,k'\in{\mathbb{N}^{\mathbb{Z}}}\atop\mbox{no assumption}}B_{akk'}^{(m_1,m_2)}\mathcal{M}_{akk'}\right).
\end{eqnarray*}

We desire to eliminate the terms $R_0,R_1$ in (\ref{N1}) by the coordinate transformation $\Phi$, which is obtained as the time-1 map $X_F^{t}|_{t=1}$ of a Hamiltonian
vector field $X_F$ with $F=F_0+F_1$. Let ${F}_{0}$ (resp. ${F}_{1}$) has the form of ${R}_0$ (resp. ${R}_{1}$),
that is \begin{eqnarray}
\label{039}&&{F}_0=\sum_{a,k,k'\in{\mathbb{N}^{\mathbb{Z}}}\atop\mbox{supp}\ k\bigcap \mbox{supp}\ k'=\emptyset}F_{akk'}\mathcal{M}_{akk'},\\
\label{040}&&{F}_1=\sum_{m\in\mathbb{Z}}J_m\left(\sum_{a,k,k'\in{\mathbb{N}^{\mathbb{Z}}}\atop\mbox{supp}\ k\bigcap \mbox{supp}\ k'=\emptyset}F_{akk'}^{(m)}\mathcal{M}_{akk'}\right),
\end{eqnarray}
and the homolnical equations become
\begin{equation}\label{041}
\{N,{F}\}+R_0+R_{1}=[R_0]+[R_1],
\end{equation}
where
\begin{equation}\label{042}
[R_0]=\sum_{a\in\mathbb{N}^{\mathbb{Z}}}B_{a00}\mathcal{M}_{a00},
\end{equation}
and
\begin{equation}\nonumber
[R_1]=\sum_{m\in\mathbb{Z}}J_m\sum_{a\in\mathbb{N}^{\mathbb{Z}}}B_{a00}^{(m)}\mathcal{M}_{a00}.
\end{equation}
The solutions of the homological equations (\ref{041}) are given by
\begin{equation}\label{044}
F_{akk'}=\frac{B_{akk'}}{\sum_{n\in\mathbb{Z}}(k_n-k^{'}_n)(n^2+\widetilde{V}_n)},
\end{equation}
and
\begin{equation}\label{045}
F_{akk'}^{(m)}=\frac{B_{akk'}^{(m)}}{\sum_{n\in\mathbb{Z}}(k_n-k^{'}_n)(n^2+\widetilde{V}_n)},
\end{equation}
where $ \widetilde{V}_n$ denote the modulated frequencies by readjusting the multiplier $(\widetilde{V}_n)$ in (\ref{maineq}) to ensure at each stage $ \widetilde{V}_n = \omega_{n}$ with $ \omega=(\omega_{n})$ a fixed frequency.\\
The new Hamiltonian ${H}_{+}$ has the form
\begin{eqnarray}
H_{+}\nonumber&=&H\circ\Phi\\
&=&\nonumber N+\{N,F\}+R_0+R_1\\
&&\nonumber+\int_{0}^1\{(1-t)\{N,F\}+R_0+R_1,F\}\circ X_F^{t}\ \mathrm{d}{t}
+\nonumber R_2\circ X_F^1\\
&=&\label{046}N_++R_+,
\end{eqnarray}
where
\begin{equation}\label{047}
N_+=N+[R_0]+[R_1],
\end{equation}
and
\begin{equation}\label{048}
R_+=\int_{0}^1\{(1-t)\{N,F\}+R_0+R_1,F\}\circ X_F^{t}\ \mathrm{d} t+R_2\circ X_F^1.
\end{equation}
\subsection{The solvability of the homolnical equations (\ref{041})}In this subsection, we will estimate
the solutions of the homolnical equations (\ref{041}). To this end, we define the new norm for the Hamiltonian ${R}$ of the form as follows:
\begin{eqnarray}\label{049}
||{R}||_{\rho,\mu}^{+}=\max\{||R_0||_{\rho,\mu}^{+},||R_1||_{\rho,\mu}^{+}|,||R_2||_{\rho,\mu}^{+}\},
\end{eqnarray}
where
\begin{eqnarray}\nonumber
&&||R_0||_{\rho,\mu}^{+}=\sup_{a,k,k'\in\mathbb{N}^{\mathbb{Z}}}\frac{\left|B_{akk'}\right|}
{e^{\rho(\sum_{n}(2a_n+k_n+k_n')\ln^{\sigma}|\lfloor n\rfloor-2\ln^{\sigma}\lfloor n_1^*\rfloor)-\mu \ln^{\sigma}\lfloor m^*(k,k')\rfloor}},\\
\nonumber&&||R_1||_{\rho,\mu}^{+}=\sup_{a,k,k'\in\mathbb{N}^{\mathbb{Z}}\atop m\in\mathbb{Z}}\frac{\left|B^{(m)}_{akk'}\right|}{e^{\rho(\sum_{n}(2a_n+k_n+k_n')\ln^{\sigma}\lfloor n\rfloor+2\ln^{\sigma}\lfloor m\rfloor-2\ln^{\sigma}\lfloor n_1^*\rfloor)
-\mu\ln^{\sigma}\lfloor m^\ast(k,k')\rfloor}},\\
\nonumber&&||R_2||_{\rho,\mu}^{+}\\
\nonumber&=&\sup_{a,k,k'\in\mathbb{N}^{\mathbb{Z}}\atop
m_1,m_2\in\mathbb{Z}}\frac{\left|B^{(m_1,m_2)}_{akk'}\right|}{e^{\rho(\sum_{n}(2a_n+k_n+k_n')\ln^{\sigma}\lfloor n\rfloor
+2\ln^{\sigma}\lfloor m_1\rfloor+2\ln^{\sigma}\lfloor m_2\rfloor-2\ln^{\sigma}\lfloor n_1^*\rfloor)-\mu \ln^{\sigma}\lfloor m^\ast(k,k')\rfloor}}.
\end{eqnarray}
Moreover, one has the following estimates:
\begin{lem}\label{N2}
Given any $ \mu>\delta>0,\rho>0$, one has
\begin{equation}\label{053}
||R||_{\rho+\delta,\mu-\delta}^{+}\leq\exp\left\{3\left(\frac{4}{\delta}\right)^{\frac{1}{\sigma-1}}\cdot
\exp\left\{\left(\frac{4}{\delta}\right)^{\frac{1}{\sigma}}\right\}\right\}||R||_{\rho,\mu}
\end{equation}
and
\begin{equation}\label{054}
||R||_{\rho+\delta,\mu-\delta}\leq\frac{64}{e^{2}\delta^2}||R||_{\rho,\mu}^{+}.
\end{equation}
\end{lem}

\begin{proof}
Firstly, we will prove the inequality (\ref{053}).
Write $\mathcal{M}_{akk'}$ in the form of
\begin{equation*}
\mathcal{M}_{akk'}=\mathcal{M}_{abll'}=\prod_nI_n(0)^{a_n}I_n^{b_n}q_n^{l_n}{\bar q_n}^{l_n'},
\end{equation*}
where
\begin{equation*}
b_n=k_n\wedge k_n',\quad l_n=k_n-b_n,\quad l_n'=k_n'-b_n'
\end{equation*}
and
$l_nl_n'=0$ for all $n$.

Express the term
\begin{equation*}
\prod_nI_n^{b_n}=\prod_n(I_n(0)+J_n)^{b_n}
\end{equation*}by the monomials of the form
\begin{equation*}
\prod_nI_n(0)^{b_n},
\end{equation*}
\begin{equation*}
\sum_{m,b_m\geq 1}\left(I_m(0)^{b_m-1}J_m\right)\left(\prod_{n\neq m}I_n(0)^{b_n}\right),
\end{equation*}
\begin{equation*}
\sum_{m,b_m\geq2\atop
r\leq b_m-2}\left(\prod_{n< m}I_n(0)^{b_n}\right)\left(I_m(0)^{r}J_m^2I_m^{b_m-r-2}\right)\left(\prod_{n> m}I_n^{b_n}\right),
\end{equation*}
and
\begin{eqnarray*}
&&\sum_{m_1< m_2,b_{m_1},b_{m_2}\geq 1\atop
r\leq b_{m_2}-1}\left(\prod_{n< m_1}I_n(0)^{b_n}\right)\left(I_{m_1}(0)^{b_{m_1}-1}J_{m_1}\right)
\\
&&\nonumber\times\left(\prod_{m_1<n< m_2}I_n(0)^{b_n}\right)\left(I_{m_2}(0)^{r}J_{m_2}I_{m_2}^{b_{m_2}-r-1}\right)
\left(\prod_{n> m_2}I_n^{b_n}\right).
\end{eqnarray*}

Now we will estimate the bounds for the coefficients respectively.

Consider the term
$\mathcal{M}_{akk'}=\prod_nI_n(0)^{a_n}q_n^{k_n}\bar q_n^{k_n'}$ with fixed $a,k,k'$ satisfying $k_nk_n'=0$ for all $n$. It is easy to see that $\mathcal{M}_{akk'}$ comes from some parts of the terms $\mathcal{M}_{\alpha\kappa\kappa'}$ with no assumption for $\kappa$ and $\kappa'$. For any given $n$ one has
\begin{equation*}
I_n(0)^{a_n}q_n^{k_n}\bar q_n^{k_n'}=\sum_{\beta_n=k_n\wedge k_n'}I_n(0)^{\alpha_n+\beta_n}q_n^{\kappa_n-\beta_n}\bar q_n^{\kappa_n'-\beta_n}.
\end{equation*}
Hence,
\begin{equation}\label{055}
\alpha_n+\beta_n=a_n,
\end{equation}
and
\begin{equation}\label{056}
\kappa_n-\beta_n=k_n,\qquad \kappa_n'-\beta_n=k_n'.
\end{equation}
Therefore, if $0\leq\alpha_n\leq a_n$ is chosen, so $\beta_n,k_n,k_n'$ are determined.
On the other hand,
\begin{eqnarray}
\nonumber&&|B_{\alpha\kappa\kappa'}|\\
&\leq&\nonumber ||R||_{\rho,\mu}e^{\rho\left(\sum_{n}(2\alpha_n+\kappa_n+\kappa_n')\ln^{\sigma}\lfloor n\rfloor-2\ln^{\sigma}\lfloor n_1^*\rfloor\right)-\mu \ln^{\sigma}\lfloor m^{\ast}(\kappa,\kappa')\rfloor}\\
&=&\nonumber||R||_{\rho,\mu}e^{\rho\left(\sum_{n}(2\alpha_n+(k_n+a_n-\alpha_n)+(k_n'+a_n-\alpha_n))\ln^{\sigma}\lfloor n\rfloor-2\ln^{\sigma}\lfloor n_1^*\rfloor\right)-\mu \ln^{\sigma}\lfloor m^{\ast}(\kappa,\kappa')\rfloor}\quad\\
&&\nonumber{(\mbox{in view of (\ref{055}) and (\ref{056})})}\\
&=&\nonumber||R||_{\rho,\mu}e^{\rho\left(\sum_{n}(2a_n+k_n+k_n')\ln^{\sigma}\lfloor n\rfloor-2\log^{\sigma}\lfloor n_1^*\rfloor\right)-\mu \log^{\sigma}\lfloor m^*(k,k')\rfloor}.
\end{eqnarray}
Hence,
\begin{equation}\label{057}
|B_{akk'}|\leq||R||_{\rho,\mu}\prod_n(1+a_n)e^{\rho\left(\sum_{n}(2a_n+k_n+k_n')\ln^{\sigma}\lfloor n\rfloor-2\ln^{\sigma}\lfloor n_1^*\rfloor\right)-\mu \ln^{\sigma}\lfloor m^*(k,k')\rfloor}.
\end{equation}
Similarly,
\begin{eqnarray*}
\left|B_{akk'}^{(m)}\right|&\leq& ||R||_{\rho,\mu}\left(\prod_{n\neq m}(1+a_n)\right)(1+a_m)^2 \nonumber\\
&&\times e^{\rho\left(\sum_{n}(2a_n+k_n+k_n')\ln^{\sigma}\lfloor n\rfloor+2\ln^{\sigma}\lfloor m\rfloor-2\ln^{\sigma}\lfloor n_1^*\rfloor\right)-\mu \ln^{\sigma}\lfloor m^*(k,k')\rfloor},\\
\left|B_{akk'}^{(m,m)}\right|&\leq& ||R||_{\rho,\mu}\left(\prod_{n\neq m}(1+a_n)\right)(1+a_m)^3 \nonumber\\ &&\times e^{\rho\left(\sum_{n}(2a_n+k_n+k_n')\ln^{\sigma}\lfloor n\rfloor
+4\ln^{\sigma}\lfloor m\rfloor-2\ln^{\sigma}\lfloor n_1^*\rfloor\right)-\mu \ln^{\sigma}\lfloor m^*(k,k')\rfloor},
\end{eqnarray*}
\begin{eqnarray*}
&&\left|B_{akk'}^{(m_1,m_2)}\right|\\
&\leq& ||R||_{\rho,\mu}\left(\prod_{n<m_1}(1+a_n)\right)(1+a_{m_1})^2\left(\prod_{m_1<n<m_2 }(1+a_n)\right)(1+a_{m_2})^2\\
&&\times e^{\rho\left(\sum_{n}(2a_n+k_n+k_n')\ln^{\sigma}\lfloor n\rfloor
+2\ln^{\sigma}\lfloor m_1\rfloor+2\ln^{\sigma}\lfloor m_2\rfloor-2\ln^{\sigma}\lfloor n_1^*\rfloor\right)-\mu \ln^{\sigma}\lfloor m^*(k,k')\rfloor}.
\end{eqnarray*}
In view of (\ref{057}), we have
\begin{eqnarray}\label{058}
&&||R_0||_{\rho+\delta,\mu-\delta}^{+}\\
\nonumber&\leq&||R||_{\rho,\mu}\prod_n(1+a_n)e^{-\delta\left(\sum_{n}(2a_n+k_n+k_n')\ln^{\sigma}\lfloor n\rfloor
-2\ln^{\sigma}\lfloor n_1^*\rfloor+\ln^{\sigma}\lfloor m^*(k,k')\rfloor\right)}.
\end{eqnarray}
Now we will show that
\begin{eqnarray}
\nonumber&&\prod_{n}(1+a_n)e^{-\delta\left(\sum_{n}(2a_n+k_n+k_n')\ln^{\sigma}\lfloor n\rfloor
-2\ln^{\sigma}\lfloor n_1^*\rfloor+\ln^{\sigma}\lfloor m^*(k,k')\rfloor\right)}\\
\label{059}&\leq& \exp\left\{3\left(\frac{4}{\delta}\right)^{\frac{1}{\sigma-1}}\cdot
\exp\left\{\left(\frac{4}{\delta}\right)^{\frac{1}{\sigma}}\right\}\right\}.
\end{eqnarray}
%where $C(\theta)$ is a positive constant depending only on $\theta$.

\textbf{Case 1.} $n_1^*=n_2^*=n_3^*.$
Then one has
\begin{eqnarray*}
\mbox{L.H.S. of~}(\ref{059})
&=&\nonumber \prod_{n}(1+a_n)e^{-{\delta}\sum_{i\geq3}\ln^{\sigma}\lfloor n_i\rfloor}e^{-\delta \ln^{\sigma}\lfloor m^*(k,k')\rfloor}\\
&\leq&\nonumber\prod_n(1+a_n)e^{-\frac{\delta}{3}\sum_{i\geq1}\ln^{\sigma}\lfloor n_i\rfloor}\\
&=&\nonumber\prod_n(1+a_n)e^{-\frac\delta3\sum_{n}(2a_n+k_n+k_n')\ln^{\sigma}\lfloor n\rfloor}\\
&\leq&\nonumber\prod_n\left((1+a_n)e^{-\frac{2\delta}{3}a_n\ln^{\sigma}\lfloor n\rfloor}\right)\\
&\leq&\exp\left\{3\left(\frac{3}{\delta}\right)^{\frac{1}{\sigma-1}}\cdot
\exp\left\{\left(\frac{3}{\delta}\right)^{\frac{1}{\sigma}}\right\}\right\},
\end{eqnarray*}
where the last equality is based on (\ref{042807}).

\textbf{Case 2.} $n_1^*>n_2^*=n_3^*.$ In this case, $a_{n}=1$ for $n=n_1$.
Then we have
\begin{eqnarray*}
\mbox{L.H.S. of~}(\ref{059})
&=&2\left(\prod_{|n|\leq n_2^*}(1+a_n)e^{-\frac{1}{2}\delta\sum_{i\geq3}\ln^{\sigma}\lfloor n_i^*\rfloor}\right)
\\
&\leq&2\prod_{|n|\leq n_2^*}(1+a_n)e^{-\frac14\delta\sum_{i\geq2}\ln^{\sigma}\lfloor n_i^*\rfloor}\\
&=&\nonumber2\prod_{|n|\leq n_2^*}(1+a_n)e^{-\frac14\delta\sum_{|n|\leq n_2^*}(2a_n+k_n+k_n')\ln^{\sigma}\lfloor n\rfloor}\\
&\leq&2\prod_{|n|\leq n_2^*}\left( (1+a_n)e^{-\frac{1}{2}\delta a_n\ln^{\sigma}\lfloor n\rfloor}\right)\\
&\leq&\exp\left\{3\left(\frac{4}{\delta}\right)^{\frac{1}{\sigma-1}}\cdot
\exp\left\{\left(\frac{4}{\delta}\right)^{\frac{1}{\sigma}}\right\}\right\},
\end{eqnarray*}
where the last equality is based on (\ref{042807}).

\textbf{Case 3.} $n_{1}^{\ast}\geq n_2^*>n_3^*.$ In this case, $a_{n}=1$ or $2$ for $n\in\{n_1, n_2\}$.
Hence
\begin{eqnarray*}
\mbox{L.H.S. of~}(\ref{059})
&\leq&4\left(\prod_{|n|\leq n_3^*}(1+a_n)e^{-\delta\sum_{i\geq3}\ln^{\sigma}\lfloor n_i^*\rfloor}\right)
\\
%&&\mbox{(in view of (\ref{}))}\\
&\leq&\nonumber4\prod_{|n|\leq n_3^*}(1+a_n)e^{-\delta\sum_{|n|\leq n_3^*}(2a_n+k_n+k_n')\ln^{\sigma}\lfloor n\rfloor}\\
&\leq&4\prod_{|n|\leq n_3^*}\left( (1+a_n)e^{-2\delta a_n\ln^{\sigma}\lfloor n\rfloor}\right)\\
&\leq& \exp\left\{3\left(\frac{1}{\delta}\right)^{\frac{1}{\sigma-1}}\cdot
\exp\left\{\left(\frac{1}{\delta}\right)^{\frac{1}{\sigma}}\right\}\right\},
\end{eqnarray*}
where the last equality is based on (\ref{042807}).\\
We finish the proof of ({\ref{059}}).

Similarly, one has
\begin{eqnarray*}
&&||R_i||_{\rho+\delta,\mu-\delta}^{+}\leq\exp\left\{3\left(\frac{4}{\delta}\right)^{\frac{1}{\sigma-1}}\cdot
\exp\left\{\left(\frac{4}{\delta}\right)^{\frac{1}{\sigma}}\right\}\right\}||R_i||_{\rho,\mu},\qquad i=1,2,
\end{eqnarray*}
and hence
\begin{equation*}
||R||_{\rho+\delta,\mu-\delta}^{+}\leq\exp\left\{3\left(\frac{4}{\delta}\right)^{\frac{1}{\sigma-1}}\cdot
\exp\left\{\left(\frac{4}{\delta}\right)^{\frac{1}{\sigma}}\right\}\right\}||R||_{\rho,\mu}.
\end{equation*}
On the other hand, the coefficient of $\mathcal{M}_{abll'}$ increases by at most a factor $$\left(\sum_{n}(a_n+b_n)\right)^2,$$ then
\begin{eqnarray}
\nonumber||R||_{\rho+\delta,\mu-\delta}
&\leq&\nonumber||R||_{\rho,\mu}^{+}\left(\sum_{n}(a_n+b_n)\right)^2\\
\nonumber &&\times e^{-\delta(\sum_{n}(2a_n+k_n+k_n')\ln^{\sigma}\lfloor n\rfloor-2\ln^{\sigma}\lfloor n_1^*\rfloor+\ln^{\sigma}\lfloor m^*(\kappa,\kappa)\rfloor)}\\
\nonumber&\leq&||R||_{\rho,\mu}^{+}\left(2\sum_{i\geq 3}\ln^{\sigma}\lfloor n_i^*\rfloor\right)^2 e^{-\frac{1}{2}\delta\sum_{i\geq3}\ln^{\sigma}\lfloor n_i^*\rfloor}\\
\label{060} &\leq&\frac{64}{e^2\delta^2}||R||_{\rho,\mu}^{+}.
\end{eqnarray}
where the last inequality is based on (\ref{042805*}) with $p=2$.

\end{proof}

\begin{lem}\label{N3}
Let $(\widetilde{V}_n)_{n\in\mathbb{Z}}$ be Diophantine with $\gamma>0$ (see (\ref{005})). Then for any $\rho>0,0<\delta\ll1$, the solutions of the homological equations (\ref{041}), which are given by (\ref{044}) and (\ref{045}), satisfy
\begin{eqnarray}\label{061}
||{F}_i||_{\rho+\delta,\mu-2\delta}^{+}\leq \frac{1}{\gamma}\cdot \exp\left\{\left(\frac{1000}{{\delta}}\right)\cdot\exp\left\{4\cdot\left(\frac{50}{\delta}
\right)^{\frac{1}{\sigma-1}}\right\}\right\}||{{R_i}}||_{\rho,\mu}^{+},
\end{eqnarray}
 where $i=0,1$.
\end{lem}

\begin{proof}

We distinguish two cases:

$\textbf{Case. 1.}$ $$\left|\sum_{n\in\mathbb{Z}}(k_n-k'_n)n^2\right|>10\sum_{n\in\mathbb{Z}}|k_n-k'_n|.$$\\
Since $|\widetilde{{V}}_n|\leq2$, we have
$$\left|\sum_{n\in\mathbb{Z}}(k_n-k'_n)(n^2+\widetilde{V}_n)\right|>10\sum_{n\in\mathbb{Z}}|k_n-k'_n|-2\sum_{n\in\mathbb{Z}}|k_n-k'_n|\geq1,$$
where the last inequality is based on $\mbox{supp}\ k\bigcap \mbox{supp}\ k'=\emptyset$.
There is no small divisor and (\ref{061}) holds trivially.

$\textbf{Case. 2.}$ $$\left|\sum_{n\in\mathbb{Z}}(k_n-k'_n)n^2\right|\leq 10\sum_{n\in\mathbb{Z}}|k_n-k'_n|.$$
In this case, we always assume  $$\left|\sum_{n\in\mathbb{Z}}(k_n-k_n')(n^2+ \widetilde{V}_n)\right|\leq1,$$otherwise there is no small divisor.

Firstly, one has
\begin{eqnarray}
\nonumber&&\sum_{n\in\mathbb{Z}}|k_n-k_n'|\ln^{\sigma}\lfloor n\rfloor\\
\nonumber&\leq&\nonumber{3\cdot 4^{\sigma}}\left(\sum_{i\geq3}\ln^{\sigma}\lfloor n_i^*\rfloor+\ln^{\sigma}\lfloor m^{*}(k,k')\rfloor\right)\qquad (\mbox{in view of Lemma \ref{a1}})\\
\label{062}&\leq& 6\cdot 4^{\sigma}\left(\sum_{n\in\mathbb{Z}}(2a_n+k_n+k_n')\ln^{\sigma}\lfloor n\rfloor-2\ln^{\sigma}\lfloor n_1^*\rfloor+2\ln^{\sigma}\lfloor m^*(k,k')\rfloor\right),
\end{eqnarray}
where the last inequality is based on Lemma \ref{H1}.

Since $\sum_{n\in\mathbb{Z}}(k_n-k'_n)n^2\in\mathbb{Z},$
the Diophantine property of $(\widetilde{V}_n)$ implies
\begin{equation}\label{063}
\left|\sum_{n\in\mathbb{Z}}(k_n-k'_n)(n^2+\widetilde{V}_n)\right|
\geq\frac\gamma2\prod_{n\in\mathbb{Z}}\frac{1}{1+{|k_n-k'_n|}^2\langle n\rangle^4}.
\end{equation}
Hence,
\begin{eqnarray}
\nonumber&&{|{F}_{akk'}|}e^{-(\rho+\delta)(\sum_{n}(2a_n+k_n+k'_n)\ln^{\sigma}\lfloor n\rfloor-2\ln^{\sigma}\lfloor n_1^{*}\rfloor)+(\mu-2\delta)\ln^{\sigma}\lfloor m^{*}(k,k')\rfloor}\\
\nonumber&\leq& 2\gamma^{-1}||{R_0}||_{\rho,\mu}^{+} \prod_{n}\left({1+{|k_n-k'_n|}^2\langle n\rangle^4}\right)\\
&&\nonumber\times e^{-\delta\left(\sum_{n}(2a_n+k_n+k'_n)\ln^{\sigma}\lfloor n\rfloor-2\ln^{\sigma}\lfloor n_1^*\rfloor+2\ln^{\sigma}\lfloor m^*(k,k')\rfloor\right)}
 \mbox{(in view of (\ref{063}))}\\
\nonumber &\leq&2\gamma^{-1} ||{R_0}||_{\rho,\mu}^+e^{\sum_{n}\ln(1+|k_n-k'_n|^2\langle n\rangle^4)}e^{-\frac{\delta}
{6\cdot 4^{\sigma}}\cdot\sum_n|k_n-k_n'|\ln^{\sigma}\lfloor n\rfloor}\\
&&\nonumber\mbox{(in view of (\ref{062}))}\\
\nonumber&=& 2\gamma^{-1} ||{R_0}||_{\rho,\mu}^+e^{\sum_{n}\ln(1+|k_n-k'_n|^2\langle n\rangle^4)}e^{-\tilde{\delta} \sum_n|k_n-k_n'|\ln^{\sigma}\lfloor n\rfloor}\ \qquad \mbox{ (note $\tilde{\delta}=\frac{\delta}{3\cdot 4^{\sigma}}$) }\\
&=&\nonumber2\gamma^{-1} ||{R_0}||_{\rho,\mu}^+e^{\sum_{{n:k_n\neq k'_n}}\ln(1+|k_n-k'_n|^2\langle n\rangle^4)-\tilde{\delta}\sum_{n:k_n\neq k'_n}|k_n-k_n'|\ln^{\sigma}\lfloor n\rfloor}\\
&\leq&\nonumber2\gamma^{-1} ||{R_0}||_{\rho,\mu}^+e^{8\left(\sum_{{n:k_n\neq k'_n}}\ln(|k_n-k'_n|\langle n\rangle)\right)+3-\tilde{\delta}\sum_{n:k_n\neq k'_n}|k_n-k_n'|\ln^{\sigma}\lfloor n\rfloor}\\
&=&\nonumber \frac{2e^{3}}{\gamma} ||{R_0}||_{\rho,\mu}^+e^{\sum_{n:k_n\neq k'_n}\left(8\ln(|k_n-k'_n|\lfloor n\rfloor)-\tilde{\delta}|k_n-k_n'|\ln^{\sigma}\lfloor n\rfloor\right)}\\
&=&\nonumber\frac{2e^{3}}{\gamma} ||{R_0}||_{\rho,\mu}^+e^{\sum_{|n|\leq N:k_n\neq k'_n}\left(8\ln(|k_n-k'_n|\lfloor n\rfloor)-\tilde{\delta}\ln^{\sigma}(|k_n-k_n'|\lfloor n\rfloor)\right)}\\
&&+\nonumber\frac{2e^{3}}{\gamma} ||{R_0}||_{\rho,\mu}^+e^{\sum_{n>N:k_n\neq k'_n}\left(8\ln(|k_n-k'_n|\lfloor n\rfloor)-\tilde{\delta}\ln^{\sigma}(|k_n-k_n'|\lfloor n\rfloor)\right)}\\
&&\nonumber \left(\mbox{where}\ N=\exp\left\{\left(\frac{8}{\tilde{\delta}\sigma }\right)^{\frac{1}{\sigma-1}}\right\} \right)\\
&=&\nonumber\frac{2e^{3}}{\gamma} ||{R_0}||_{\rho,\mu}^+\exp\left\{16\cdot\left(\frac{48}{{\delta}}\right)^{\frac{1}{\sigma-1}}
\cdot\exp\left\{4\cdot\left(\frac{48}{{\delta}}\right)^{\frac{1}{\sigma-1}}\right\}\right\}
+\frac{2e^{3}}{\gamma} ||{R_0}||_{\rho,\mu}^+\\
\nonumber&&\mbox{(in view of  (\ref{042805}))}\\
&\leq&\label{064}\frac{1}{\gamma}\cdot\exp\left\{20\cdot\left(\frac{48}{{\delta}}\right)^{\frac{1}{\sigma-1}}
\cdot\exp\left\{4\cdot\left(\frac{48}{{\delta}}\right)^{\frac{1}{\sigma-1}}\right\}\right\}||{R_0}||_{\rho,\mu}^+.
\end{eqnarray}
Therefore, in view of (\ref{064}), we finish the proof of (\ref{061}) for $i=0$.
Similarly, one can prove (\ref{061}) for $i=1,2$.
\end{proof}

\subsection{The new perturbation $R_+$ and the new normal form $N_+$}Firstly, we will prove two lemmas.

Recall the new term $R_+$ is given by (\ref{048}) and
write
\begin{equation}\label{069}
R_+=R_{0+}+R_{1+}+R_{2+}.
\end{equation}

Following the proof of \cite{CLSY}, one has
\begin{eqnarray}
||R_{0+}||_{\rho+3\delta,\mu-\frac{11}{2}\delta}^{+}
\label{0861} &\leq& \frac1{\gamma}\cdot \exp\left\{\frac{1000}{{\delta}}\exp\left\{4\cdot\left(\frac{100}{\delta}
\right)^{\frac{1}{\sigma-1}}\right\}\right\}\\
\nonumber&&\times(||R_0||_{\rho,\mu}^++||R_1||_{\rho,\mu}^+)(||R_0||_{\rho,\mu}^+
+{||R_1||_{\rho,\mu}^+}^2), \\
||R_{1+}||_{\rho+3\delta,\mu-\frac{11}{2}\delta}^{+}
\label{0862}&\leq& \frac1{\gamma}\cdot \exp\left\{\frac{1000}{{\delta}}\exp\left\{4\cdot\left(\frac{100}{\delta}
\right)^{\frac{1}{\sigma-1}}\right\}\right\}\\
\nonumber&&\times(||R_0||_{\rho,\mu}^++{||R_1||_{\rho,\mu}^+}^2), \\
||R_{2+}||_{\rho+3\delta,\mu-\frac{11}{2}\delta}^{+}
\label{0863}&\leq& ||R_2||_{\rho,\mu}^++\frac1{\gamma}\cdot\exp\left\{\frac{1000}{{\delta}}\exp\left\{4\cdot\left(\frac{100}{\delta}
\right)^{\frac{1}{\sigma-1}}\right\}\right\}\\
\nonumber&&\times(||R_0||_{\rho,\mu}^++||R_1||_{\rho,\mu}^+).
\end{eqnarray}

The new normal form $N_+$ is given in (\ref{047}). Note that $[R_0]$ (in view of (\ref{042})) is a constant which does not affect the Hamiltonian vector field. Moreover, in view of (\ref{042}), we denote by
\begin{equation}\label{087}
\omega_{n+}=n^2+ \widetilde{V}_n+\sum_{a\in\mathbb{N}^{\mathbb{Z}}}B_{a00}^{(n)}\mathcal{M}_{a00},
\end{equation}
where the terms $\sum_{a\in\mathbb{N}^{\mathbb{Z}}}B_{a00}^{(n)}\mathcal{M}_{a00}$ is the so-called frequency shift. The estimate of $\left|\sum_{a\in\mathbb{N}^{\mathbb{Z}}}B_{a00}^{(n)}\mathcal{M}_{a00}\right|$ will be given in the next section (see (\ref{114}) for the details).

\section{Iteration and Convergence}

Now we give the precise
set-up of iteration parameters. Let $s\geq1$ be the $s$-th KAM
step.
 \begin{itemize}
 \item[] $\rho_0=\frac{3-2\sqrt{2}}{100},$\ $r\geq 200\rho_{0}$, $\mu_0\geq2r$,
 \item[]$\delta_{s}=\frac{\rho_{0}}{(s+4)\ln^2(s+4)}$,

 \item[]$\rho_{s+1}=\rho_{s}+3\delta_s$,\item[] $\mu_{s+1}=\mu_{s}-6\delta_s$

 \item[]$\epsilon_s=\epsilon_{0}^{(\frac{3}{2})^s}$, which dominates the size of
 the perturbation,

 \item[]$\lambda_s=\epsilon^{0.01}_{s}$,

 \item[]$\eta_{s+1}=\frac{1}{20}\lambda_s\eta_s$,

 \item[]$d_0=0,\,d_{s+1}=d_s+\frac{1}{\pi^2(s+1)^2}$,

 \item[]$D_s=\{(q_n)_{n\in\mathbb{Z}}:\frac{1}{2}+d_s\leq|q_n|e^{r \ln^{\sigma}\lfloor n\rfloor}\leq1-d_s\}$.
 \end{itemize}
Denote the complex cube of size $\lambda>0$:
\begin{equation}\label{088}
\mathcal{C}_{\lambda}({V^*})=\left\{(V_n)_{n\in\mathbb{Z}}\in\mathbb{C}^{\mathbb{Z}}:|V_n-V^*_n|\leq \lambda\right\}.
\end{equation}

\begin{lem}{\label{E2}}
Suppose $H_{s}=N_{s}+R_{s}$ is real analytic on $D_{s}\times\mathcal{C}_{\eta_{s}}(V^*_{s})$,
where $$N_{s}=\sum_{n\in\mathbb{Z}}(n^2+\widetilde V_{n,s})|q_n|^2$$ is a normal form with coefficients satisfying
\begin{eqnarray}
\label{089}&&\widetilde{V}_{s}(V_{s}^*)=\omega,\\
\label{090}&&\left|\left|\frac{\partial \widetilde{V}_s}{{\partial V}}-I\right|\right|_{l^{\infty}\rightarrow l^{\infty}}<d_s\epsilon_{0}^{\frac{1}{10}},
\end{eqnarray}
and $R_{s}=R_{0,s}+R_{1,s}+R_{2,s}$ satisfying
\begin{eqnarray}
\label{091}&&||R_{0,s}||_{\rho_{s},\mu_{s}}^{+}\leq \epsilon_{s},\\
\label{092}&&||R_{1,s}||_{\rho_{s},\mu_{s}}^{+}\leq \epsilon_{s}^{0.6},\\
\label{093}&&||R_{2,s}||_{\rho_{s},\mu_{s}}^{+}\leq (1+d_s)\epsilon_0.
\end{eqnarray}
Then for all $V\in\mathcal{C}_{\eta_{s}}(V_{s}^*)$ satisfying $\widetilde V_{s}(V)\in\mathcal{C}_{\lambda_s}(\omega)$, there exist real analytic symplectic coordinate transformations
$\Phi_{s+1}:D_{s+1}\rightarrow D_{s}$ satisfying
\begin{eqnarray}
\label{094}&&||\Phi_{s+1}-id||_{r,\infty}\leq \epsilon_{s}^{0.5},\\
\label{095}&&||D\Phi_{s+1}-I||_{(r,\infty)\rightarrow(r,\infty)}\leq \epsilon_{s}^{0.5},
\end{eqnarray}
such that for
$H_{s+1}=H_{s}\circ\Phi_{s+1}=N_{s+1}+R_{s+1}$, the same assumptions as above are satisfied with `$s+1$' in place of `$s$', where $\mathcal{C}_{\eta_{s+1}}(V_{s+1}^*)\subset\widetilde V_{s}^{-1}(\mathcal{C}_{\lambda_s}(\omega))$ and
\begin{equation}\label{096}
||\widetilde{V}_{s+1}-\widetilde{V}_{s}||_{\infty}\leq\epsilon_{s}^{0.5},
\end{equation}
\begin{equation}\label{097}
||V_{s+1}^*-V_{s}^*||_{\infty}\leq2\epsilon_{s}^{0.5}.
\end{equation}
\end{lem}

\begin{proof}
In the step $s\rightarrow s+1$, there is saving of a factor
\begin{equation}\label{19010501}
e^{-\delta_{s}\left(\sum_{n}(2a_n+k_n+k'_n)\ln^{\sigma}\lfloor n\rfloor-2\ln^{\sigma}\lfloor n_1^*\rfloor+2\ln^{\sigma}\lfloor m^{\ast}(k,k')\rfloor\right)}.
\end{equation}
By (\ref{H2}), one has
\begin{equation}\nonumber
(\ref{19010501})\leq e^{-\frac{1}{2}\delta_{s}\left(\sum_{i\geq3}\ln^{\sigma}\lfloor n_i\rfloor \right)-\delta_s\ln^{\sigma}\lfloor m^*(k,k')\rfloor}\leq e^{-\frac{1}{2}\delta_{s}\left(\sum_{i\geq3} \ln^{\sigma}\lfloor n_i\rfloor+\ln^{\sigma}\lfloor m^*(k,k')\rfloor\right)}.
\end{equation}
Recalling after this step, we need
\begin{eqnarray*}
&&||R_{0,s+1}||_{\rho_{s+1},\mu_{s+1}}^{+}\leq \epsilon_{s+1},\\
&&||R_{1,s+1}||_{\rho_{s+1},\mu_{s+1}}^{+}\leq \epsilon_{s+1}^{0.6}.
\end{eqnarray*}
Consequently, in $R_{i,s}\ (i=0,1)$, it suffices to eliminate the non-resonant monomials $\mathcal{M}_{akk'}$ for which
\begin{equation*}
e^{-\frac{1}{2}\delta_{s}(\sum_{i\geq3}\ln^{\sigma}\lfloor n_i\rfloor+\ln^{\sigma}\lfloor m^{\ast}(k,k')\rfloor)}\geq\epsilon_{s+1},
\end{equation*}
that is
\begin{equation}\label{098}
\sum_{i\geq3}\ln^{\sigma}\lfloor n_i\rfloor +\ln^{\sigma}\lfloor m^{\ast}(k,k')\rfloor\leq
\frac{2(s+4)\ln^{2}(s+4)}{\rho_{0}}\ln\frac{1}{\epsilon_{s+1}}.
\end{equation}
On the other hand, in the small divisors analysis (see Lemma \ref{a1}), one has
\begin{eqnarray}
\nonumber\sum_{n\in\mathbb{Z}}|k_n-k_n'|\ln^{\sigma}\lfloor n\rfloor
\nonumber&\leq& 3\cdot 4^{\sigma}\cdot\left(\sum_{i\geq3}\ln^{\sigma}\lfloor n_i\rfloor+\ln^{\sigma}\lfloor m^{\ast}(k,k')\rfloor\right)\ \  \\
&\leq&\nonumber 3\cdot 4^{\sigma}\cdot\frac{2(s+4)\ln^{2}(s+4)}{\rho_{0}}\ln\frac{1}{\epsilon_{s+1}}\ \  \mbox{(in view of  (\ref{098}))}\\
\nonumber&:=& B_s.
\end{eqnarray}
Hence we need only impose condition on $(\widetilde{V}_n)_{|n|\leq \mathcal{N}_{s}}$, where
\begin{equation}\label{100}
\mathcal{N}_{s}\sim \exp \left\{B_s^{\frac{1}{\sigma}}\right\}.
\end{equation}
Correspondingly, the Diophantine condition becomes
\begin{equation}\label{101}
\left|\left|\sum_{|n|\leq \mathcal{N}_{s}}(k_n-k'_n)\widetilde{V}_{n,s}\right|\right|\geq \gamma\prod_{|n|\leq \mathcal{N}_{s}}\frac{1}{1+(k_n-k'_n)^2\langle n\rangle^4}.
\end{equation}
We finished the truncation step.

Next we will show (\ref{101})  preserves under small perturbation of $(\widetilde{V}_n)_{|n|\leq \mathcal{N}_{s}}$ and this is equivalent to get lower bound on the right hand side of (\ref{101}). Let
\begin{equation}\nonumber
M_s\sim \left(\frac{B_s^{\frac{\sigma-1}{\sigma}}}{\ln^{\sigma} B_s}\right),
\end{equation} then we have
\begin{eqnarray}
\nonumber&&\prod_{|n|\leq \mathcal{N}_{s}}\frac{1}{1+(k_n-k'_n)^2\langle n\rangle^4}\\
\nonumber&=&\exp\left\{\sum_{|n|\leq M_s}\ln\left(\frac{1}{1+(k_n-k'_n)^2\langle n\rangle^4}\right)+\sum_{|n|> M_s}\ln\left(\frac{1}{1+(k_n-k'_n)^2\langle n\rangle^4}\right)\right\}\\
\nonumber&\geq& \exp\left\{-10\sum_{|n|\leq M_s,k_n\neq k'_n}(\ln(|k_n-k'_n|+8))\ln\lfloor n\rfloor\right.\\
\nonumber&&\left.-\sum_{|n|> M_s,k_n\neq k'_n}10\left(|k_n-k'_n|\ln\lfloor n\rfloor\right)\right\}\\
\nonumber&\geq& \exp\left\{-60M_s\ln B_s^{\frac{1}{\sigma}}-10B_s(\ln M_s)^{1-\sigma}\right\},\  \mbox{(in view of (\ref{098}))}\\
\nonumber&\geq& \exp\left\{-100B_s(\ln B_s)^{1-\sigma}\right\}\\
\label{103}&>&\exp\left\{-0.01\cdot\ln{\frac{1}{\epsilon_{s}}}\right\}=\lambda_s,
\end{eqnarray}
where the last inequality is based on $\sigma > 2$ and $\epsilon_0$ is small enough depending $ \sigma$ only.

Assuming $V\in \mathcal{C}_{\lambda_s}(\omega)$, from the lower bound (\ref{103}), the relation (\ref{101}) remains true if we substitute $V$ for $\omega$. Moreover, there is analyticity on $\mathcal{C}_{\lambda_s}(\omega)$. The transformations $\Phi_{s+1}$ is obtained as the time-1 map $X_{F_s}^{t}|_{t=1}$ of the Hamiltonian
vector field $X_{F_s}$ with $F_s=F_{0,s}+F_{1,s}$. Taking $\rho=\rho_s$, $\delta=\delta_s$ in Lemma \ref{N3}, we get
\begin{eqnarray}\label{104}
||F_{i,s}||_{\rho_s+\delta_s,\mu_s-{2}\delta_s}^{+}\leq \frac{1}{\gamma}\cdot \epsilon^{-0.01}_{s}||R_{i,s}||_{\rho_s,\mu_s}^{+},
\end{eqnarray}
where $i=0,1$. \\
By Lemma \ref{N2}, we get
\begin{equation}\label{105}
||F_{i,s}||_{\rho_s+2\delta_s,\mu_s-3\delta_s}
\leq\frac{64}{e^{2}\delta_s^2}||F_{i,s}||_{\rho_s+\delta_s,\mu_s-2\delta_s}^{+}\leq \epsilon^{0.95}_{s}.
\end{equation}
Combining (\ref{091}), (\ref{092}), (\ref{104}) and (\ref{105}), we get
\begin{equation}\label{106}
||F_{s}||_{\rho_s+2\delta_s,\mu_s-3\delta_s}\leq\epsilon_{s}^{0.58}.
\end{equation}
By Lemma \ref{H6}, we get
\begin{eqnarray}
\sup_{||q||_{r,\infty}<1}||X_{F_s}||_{r,\infty}
\label{107}&\leq&\epsilon_{s}^{0.55}.
\end{eqnarray}

Since $\epsilon_{s}^{0.55}\ll\frac{1}{\pi^2(s+1)^2}=d_{s+1}-d_s$, we have $\Phi_{s+1}:D_{s+1}\rightarrow D_{s}$ with
\begin{equation}\label{108}
\|\Phi_{s+1}-id\|_{r,\infty}\leq\sup_{q\in D_s}\|X_{F_s}\|_{r,\infty}\leq\epsilon_{s}^{0.55}<\epsilon_{s}^{0.5},
\end{equation}
which is the estimate (\ref{094}). Moreover, from (\ref{108}) we get
\begin{equation}\label{109}
\sup_{q\in D_s}||DX_{F_s}-I||_{r,\infty}\leq\frac{1}{d_s}\epsilon_{s}^{0.55}\ll\epsilon_{s}^{0.5},
\end{equation}
and thus the estimate (\ref{095}) follows.

Moreover, under the assumptions (\ref{091})--(\ref{093}) at stage $s$, we get from (\ref{0861}), (\ref{0862}) and (\ref{0863}) that
\begin{eqnarray*}
||R_{0,s+1}||_{\rho_{s+1},\mu_{s+1}}^{+}
&\leq& \frac1{\gamma}\cdot\exp\left\{\frac{1000}{\frac{\rho_{0}}{(s+4)\ln^{2}(s+4)}}
\exp\left\{4\cdot\left(\frac{100}{\frac{\rho_{0}}{(s+4)\ln^{2}(s+4)}}
\right)^{\frac{1}{\sigma-1}}\right\}\right\}\\
\nonumber&&\times\left(\epsilon_{0}^{(\frac{3}{2})^s}+\epsilon_{0}^{0.9(\frac{3}{2})^{s-1}}\right)\left(\epsilon_{0}^{(\frac{3}{2})^s}+\epsilon_{0}^{1.8(\frac{3}{2})^{s-1}}\right)\\
&=& \epsilon^{-0.01}_{s}\left(\epsilon_{0}^{2.2(\frac{3}{2})^s}+\epsilon_{0}^{1.8(\frac{3}{2})^s}
+\epsilon_{0}^{1.6(\frac{3}{2})^s}+\epsilon_{0}^{2(\frac{3}{2})^s}\right)\\
\nonumber&&\mbox{for $0<\epsilon_0\ll1$ (depending on $\sigma$ only)}\\
&=&\epsilon_{s+1},\\
||R_{1,s+1}||_{\rho_{s+1},\mu_{s+1}}^{+}
&\leq& \frac1{\gamma}\cdot\exp\left\{\frac{1000}{\frac{\rho_{0}}{(s+4)\ln^{2}(s+4)}}
\exp\left\{4\cdot\left(\frac{100}{\frac{\rho_{0}}{(s+4)\ln^{2}(s+4)}}
\right)^{\frac{1}{\sigma-1}}\right\}\right\}\\
\nonumber&&\times \left( \epsilon_{0}^{(\frac{3}{2})^s}+\epsilon_{0}^{1.8(\frac{3}{2})^{s-1}}\right)\\
&\leq& \epsilon^{-0.01}_{s}\left( \epsilon_{0}^{(\frac{3}{2})^s}+\epsilon_{0}^{1.2(\frac{3}{2})^{s}}\right) \\
&<&\epsilon_{s+1}^{0.6}\ \ \mbox{for $0<\epsilon_0\ll1$ (depending on $\sigma$ only)},\\
\end{eqnarray*}
and
\begin{eqnarray*}
&&||R_{2,s+1}||_{\rho_{s+1},\mu_{s+1}}^{+}\\
\nonumber&\leq& ||R_{2,s}||_{\rho_{s},\mu_{s}}^{+}+\frac1{\gamma}\cdot\exp\left\{\frac{1000}{\frac{\rho_{0}}{(s+4)\ln^{2}(s+4)}}
\exp\left\{4\cdot\left(\frac{100}{\frac{\rho_{0}}{(s+4)\ln^{2}(s+4)}}
\right)^{\frac{1}{\sigma-1}}\right\}\right\}\\
\nonumber&&\times
\left(\epsilon_{0}^{(\frac{3}{2})^s}+\epsilon_{0}^{0.6(\frac{3}{2})^{s}}\right)\\
&\leq&(1+d_s)\epsilon_0+ \epsilon^{-0.01}_{s}\cdot\epsilon_{0}^{0.6(\frac{3}{2})^s}\\
&\leq&(1+d_{s+1})\epsilon_0\ \ \mbox{for $0<\epsilon_0\ll1$ (depending on $\sigma$ only)},
\end{eqnarray*}
which are just the assumptions (\ref{091})--(\ref{093}) at stage $s+1$.

If $V\in \mathcal{C}_{\frac{\eta_s}{2}}(V_s^*)\subset\mathcal{C}_{{\eta_s}}(V_s^*)$ and using Cauchy's estimate, for any $m$ one has
\begin{eqnarray}
\nonumber\sum_{n\in\mathbb{Z}}\left|\frac{\partial \widetilde{V}_{m,s}}{\partial V_n}(V)\right|
\nonumber&\leq& \frac{2}{\eta_s}||\widetilde{V}_s||_\infty\\
\label{110}&<&10 \eta_s^{-1}\ \ \mbox{(since $||\widetilde{V}_s||_\infty\leq 1 $)}.
\end{eqnarray}
Let $V\in \mathcal{C}_{\frac{1}{10}\lambda_s\eta_s}(V_s^*)$, then
\begin{eqnarray*}
||\widetilde{V}_s(V)-\omega||_{\infty}
&=&||\widetilde{V}_s(V)-\widetilde{V}_s(V_s^*)||_{\infty}\\
& \leq&\sup_{\mathcal{C}_{\frac{1}{10}\lambda_s\eta_s}(V_s)}\left|\left|\frac{\partial \widetilde{V}_s}{\partial V}\right|\right|_{l^{\infty}\rightarrow l^{\infty}}||V-V_s^*||_{\infty}\\
&<&10 \eta_s^{-1}\cdot\frac{1}{10}\lambda_s\eta_s\ \ \mbox{(in view of (\ref{110}))}\\
&=&\lambda_s,
\end{eqnarray*}
that is
\begin{equation*}
\widetilde{V}_s\left(\mathcal{C}_{\frac{1}{10}\lambda_s\eta_s}(V_s)\right)\subseteq \mathcal{C}_{\lambda_s}(\omega).
\end{equation*}
Note that
\begin{eqnarray}
\nonumber\left|B^{(m)}_{a00}\right|
\nonumber&\leq& ||R_{1,s+1}||_{\rho_{s+1},\mu_{s+1}}^+e^{2\rho_{s+1}\left(\sum_{n}a_n\ln^{\sigma}\lfloor n\rfloor+\ln^{\sigma}\lfloor m\rfloor-\ln^{\sigma}\lfloor n_1^{*}\rfloor\right)}\\
\label{111}&<&\epsilon_{0}^{0.6(\frac{3}{2})^{s}}e^{2\rho_{s+1}
\left(\sum_{n}a_n\ln^{\sigma}\lfloor n\rfloor+\ln^{\sigma}\lfloor m\rfloor-\ln^{\sigma}\lfloor n_1^{*}\rfloor\right)}.
\end{eqnarray}
Assuming further
\begin{equation}\label{112}
I_{n}(0)\leq e^{-2r\ln^{\sigma}\lfloor n\rfloor}
\end{equation}
and for any $s$,
\begin{equation}\label{113}
\rho_s<\frac{1}{2}r,
\end{equation}
we obtain
\begin{eqnarray}
\nonumber&&\left|\sum_{a\in\mathbb{N}^{\mathbb{Z}}}B^{(m)}_{a00}\mathcal{M}_{a00}\right|\\
\nonumber&\leq & \epsilon_{0}^{0.6(\frac{3}{2})^{s}}\sum_{a\in\mathbb{N}^{\mathbb{Z}}}e^{2\rho_{s+1}
\left(\sum_{n}a_n\ln^{\sigma}\lfloor n\rfloor+\ln^{\sigma}\lfloor m\rfloor-\ln^{\sigma}\lfloor n_1^{*}\rfloor\right)}\prod_{n\in\mathbb{Z}}I_{n}(0)^{a_n}\\
\nonumber&\leq& \epsilon_{0}^{0.6(\frac{3}{2})^{s}}\sum_{a\in\mathbb{N}^{\mathbb{Z}}}e^{2\rho_{s+1}
\left(\sum_{n}a_n\ln^{\sigma}\lfloor n\rfloor\right)}\prod_{n\in\mathbb{Z}}I_{n}(0)^{a_n}\\
\nonumber&\leq& \epsilon_{0}^{0.6(\frac{3}{2})^{s}}\sum_{a\in\mathbb{N}^{\mathbb{Z}}}e^{\sum_{n}2\rho_{s+1}a_n\ln^{\sigma}\lfloor n\rfloor-\sum_{n}2r a_n\ln^{\sigma}\lfloor n\rfloor}\ \mbox{(in view of (\ref{112}))}\\
\nonumber&\leq& \epsilon_{0}^{0.6(\frac{3}{2})^{s}}\sum_{a\in\mathbb{N}^{\mathbb{Z}}}e^{-r\left(\sum_{n}a_n\ln^{\sigma}\lfloor n\rfloor\right)}\ \mbox{(in view of (\ref{113}))}\\
\nonumber&\leq& \epsilon_{0}^{0.6(\frac{3}{2})^{s}}\prod_{n\in\mathbb{Z}}\left(1-e^{-r \ln^{\sigma}\lfloor n\rfloor}\right)^{-1}\  \ \mbox{(in view of (\ref{041809}))}\\
\label{114}&\leq& \exp\left\{\left(\frac{18}{r}\right)\cdot \exp\left\{{\left(\frac{4}r \right)}^{\frac{1}{\sigma-1}}\right\}\right\}\epsilon_{0}^{0.6(\frac{3}{2})^{s}},\ \ \mbox{(in view of (\ref{122401}))}.
\end{eqnarray}
By (\ref{114}), we have
\begin{eqnarray}
\nonumber\left|\widetilde{V}_{m,s+1}-\widetilde{V}_{m,s}\right|
\nonumber&<& \exp\left\{\left(\frac{18}{r}\right)\cdot \exp\left\{{\left(\frac{4}r \right)}^{\frac{1}{\sigma-1}}\right\}\right\}\epsilon_{0}^{0.6(\frac{3}{2})^{s}}\\
\label{115}&<&\epsilon_{s}^{0.5}\ \ \mbox{(for $\epsilon_0$ small enough)},
\end{eqnarray}
which verifies (\ref{096}). Further applying Cauchy's estimate on $\mathcal{C}_{\lambda_s\eta_s}(V_s^*)$, one gets
\begin{eqnarray}\label{116}
\sum_{n\in\mathbb{Z}}\left|\frac{\partial \widetilde{V}_{m,s+1}}{\partial V_n}-\frac{\partial \widetilde{V}_{m,s}}{\partial V_n}\right|
\leq \frac{||\widetilde{V}_{s+1}-\widetilde{V}_{s}||_\infty}{\lambda_s\eta_s}\leq \frac{\epsilon_{s}^{0.5}}{\lambda_s\eta_s}.
\end{eqnarray}
Since
\begin{equation*}
\eta_{s+1}=\frac{1}{20}\lambda_s\eta_s,
\end{equation*}
hence one has
\begin{eqnarray}
\nonumber\lambda_s\eta_{s}&\geq& 20\eta_{0}\prod^{s}_{0} \frac{1}{20}\lambda_i\\
\nonumber&\geq&  20\eta_{0}\prod^{s}_{0} \frac{1}{20}\epsilon^{0.01}_i\\
\nonumber&\geq&  20\eta_{0}\prod^{s}_{0}\epsilon^{0.02}_i\\
\label{117}&\geq &20\epsilon_0^{\frac{3}{50}\times(\frac32)^{s}}.
\end{eqnarray}
On $ \mathcal{C}_{\frac{1}{10}\lambda_s\eta_s}(V_s^*)$ and for any $m$, we deduce from (\ref{116}), (\ref{117}) and the assumption (\ref{090}) that
\begin{eqnarray*}
\sum_{n\in\mathbb{Z}}\left|\frac{\partial \widetilde{V}_{m,s+1}}{\partial V_n}-\delta_{mn}\right|
&\leq&\sum_{n\in\mathbb{Z}}\left|\frac{\partial \widetilde{V}_{m,s+1}}{\partial V_n}-\frac{\partial \widetilde{V}_{m,s}}{\partial V_n}\right|+\sum_{n\in\mathbb{Z}}\left|\frac{\partial \widetilde{V}_{m,s}}{\partial V_n}-\delta_{mn}\right|\\
&\leq&\epsilon_{0}^{\frac{22}{50}\times(\frac{3}{2})^{s}}+d_s\epsilon_{0}^{\frac{1}{10}}\\
&<&d_{s+1}\epsilon_{0}^{\frac{1}{10}},
\end{eqnarray*}
and consequently
\begin{equation}\label{119}
\left|\left|\frac{\partial \widetilde{V}_{s+1}}{{\partial V}}-I\right|\right|_{l^{\infty}\rightarrow l^{\infty}}<d_{s+1}\epsilon_{0}^{\frac{1}{10}},
\end{equation}
which verifies (\ref{090}) for $s+1$.

Finally, we will freeze $\omega$ by invoking an inverse function theorem. From (\ref{119}) and the standard inverse function theorem,\ we can see that the functional equation
\begin{equation}\label{120}
\widetilde{V}_{s+1}(V_{s+1}^*)=\omega,  V_{s+1}^*\in \mathcal{C}_{\frac{1}{10}\lambda_s\eta_s}(V_s^*),
\end{equation}
has a solution $V_{s+1}^*$, which verifies (\ref{089}) for $s+1$. Rewriting (\ref{120}) as
\begin{equation}\label{121}
V_{s+1}^*-V_s^*=(I-\widetilde{V}_{s+1})(V^*_{s+1})-(I-\widetilde{V}_{s+1})({V^*_s})+(\widetilde{V}_s-\widetilde{V}_{s+1})(V_s^*),
\end{equation}
and by using (\ref{115}), (\ref{119}) implies
\begin{equation}\label{122}
||V_{s+1}^*-V_s^*||_{\infty}\leq (1+d_{s+1})\epsilon_{0}^{\frac{1}{10}}||V_{s+1}^*-V_s^*||_{\infty}+\epsilon_s^{0.5}<2\epsilon_s^{0.5}\ll \lambda_s\eta_s,
\end{equation}
which verifies (\ref{097}) and completes the proof of the iterative lemma.
\end{proof}

We are now in a position to prove the convergence. To apply iterative lemma with $s=0$, set
\begin{equation*}
V_0=\omega,\hspace{12pt}\widetilde{V}_0=id,\hspace{12pt}\eta_0=1-\sup_{n\in\mathbb{Z}}|\omega_n|,
\hspace{12pt}r=200\rho_0,\hspace{12pt}\mu_0=2r,\hspace{12pt}\epsilon_0=C\epsilon,
\end{equation*}
and consequently (\ref{089})--(\ref{093}) with $s=0$ are satisfied. Hence, the iterative lemma applies, and we obtain a decreasing
sequence of domains $D_{s}\times\mathcal{C}_{\eta_{s}}(V_{s}^*)$ and a sequence of
transformations
\begin{equation*}
\Phi^s=\Phi_1\circ\cdots\circ\Phi_s:\hspace{6pt}D_{s}\times\mathcal{C}_{\eta_{s}}(V_{s}^*)\rightarrow D_{0}\times\mathcal{C}_{\eta_{0}}(V_{0}^*),
\end{equation*}
such that $H\circ\Phi^s=N_s+P_s$ for $s\geq1$. Moreover, the
estimates (\ref{094})--(\ref{097}) hold. Thus we can show $V_s^*$ converge to a limit $V_*$ with the estimate
\begin{equation*}
||V_*-\omega||_{\infty}\leq\sum_{s=0}^{\infty}2\epsilon_{s}^{0.5}<\epsilon_{0}^{0.4},
\end{equation*}
and $\Phi^s$ converge uniformly on $D_*\times\{V_*\}$, where $D_*=\{(q_n)_{n\in\mathbb{Z}}:\frac{2}{3}\leq|q_n|e^{r \ln^{\sigma}\lfloor n\rfloor}\leq\frac{5}{6}\}$, to $\Phi:D_*\times\{V_*\}\rightarrow D_0$ with the estimates
\begin{eqnarray}
\nonumber&&||\Phi-id||_{r,\infty}\leq \epsilon_{s}^{0.4},\\
\nonumber&&||D\Phi-I||_{(r,\infty)\rightarrow(r,\infty)}\leq \epsilon_{s}^{0.4}.
\end{eqnarray}
Hence
\begin{equation}\label{123}
H_*=H\circ\Phi=N_*+R_{2,*},
\end{equation}
where
\begin{equation}\label{124}
N_*=\sum_{n\in\mathbb{Z}}(n^2+\omega_n)|q_n|^2
\end{equation}
and
\begin{equation}\label{125}
||R_{2,*}||_{\frac{r}{2},\frac{r}{2}}^{+}\leq\frac{7}{6}\epsilon_0.
\end{equation}
By (\ref{029}), the Hamiltonian vector field $X_{R_{2,*}}$ is a bounded map from $\mathfrak{H}_{r,\infty}$ into $\mathfrak{H}_{r,\infty}$. Taking
\begin{equation}\label{126}
I_n(0)=\frac{3}{4}e^{-2r\ln^{\sigma}\lfloor n\rfloor},
\end{equation}
we get an invariant torus $\mathcal{T}$ with frequency $(n^2+\omega_n)_{n\in\mathbb{Z}}$ for ${X}_{H_*}$. Finally, by $X_H\circ\Phi=D\Phi\cdot{X}_{H_*}$, $\Phi(\mathcal{T})$ is the desired invariant torus for the NLS (\ref{maineq}). Moreover, we deduce the torus $\Phi(\mathcal{T})$ is linearly stable from the fact that (\ref{123}) is a normal form of order 2 around the invariant torus.

\section{Appendix}
\subsection{Technical Lemmas}

\begin{lem}\label{122203}
Given any $\sigma>2$, there exists a constant $c(\sigma)>1 $ depending on $\sigma$ only such that
\begin{equation}\label{122201}
\ln^{\sigma}(x+y)-\ln^{\sigma}x-\frac12\ln^{\sigma}y\leq 0, \quad \mbox{for}\ c(\sigma)\leq y\leq x.
\end{equation}
\end{lem}
\begin{proof}
The proof of this lemma see the proof of Lemma 4.1 in \cite{cong2024}.
\end{proof}

\begin{lem}\label{a1}Let $\theta\in(0,1)$ and $k_n,k'_n\in\mathbb{N},|{\widetilde{V}}_n|\leq2\ \mbox{for}\ \forall \ n\in\mathbb{Z}$.
Assume further
\begin{equation}
\label{0004}\left|\sum_{n\in\mathbb{Z}}(k_n-k_n')(n^2+ \widetilde{V}_n)\right|\leq1.
\end{equation}
Then one has
\begin{equation}\label{0006}
\sum_{n\in\mathbb{Z}}|k_n-k_n'|\ln^{\sigma}\lfloor n\rfloor \leq 3 \cdot 4^{\sigma}\left(\sum_{i\geq3}\ln^{\sigma}\lfloor n_i\rfloor+ \ln^{\sigma}\lfloor m^{*}(k,k')\rfloor\right),
\end{equation}
where $(n_i)_{i\geq1}, |n_1|\geq|n_2|\geq|n_3|\geq\cdots$, denote the system \{$n$: $n$ is repeated $k_n+k'_n$ times\}.
\end{lem}
\begin{proof}
From the definition of $(n_i)_{i\geq1}$, there exists $(\mu_i)_{i\geq1}$ with $\mu_i\in\{-1,1\}$ such that
\begin{equation}
\label{0007}m(k,k')=\sum_{i\geq1}\mu_in_i,
\end{equation}
and
\begin{equation}
\label{0008}\sum_{n\in\mathbb{Z}}(k_n-k_n')n^2=\sum_{i\geq1}\mu_in_i^2.
\end{equation}
In view of (\ref{0004}), (\ref{0008}) and $|\widetilde{V}_n|\leq 2$,
one has
\begin{equation*}
\left|\sum_{i\geq1}\mu_in_i^2\right|\leq\left|\sum_{n\in\mathbb{Z}}(k_n-k_n') \widetilde{V}_n\right|+1\leq2\sum_{n\in\mathbb{Z}}(k_n+k_n')+1,
\end{equation*}which implies
\begin{equation}\label{0009}
\left|n_1^2+\left(\frac{\mu_2}{\mu_1}\right)n_2^2\right|\leq2\sum_{i\geq1}1+\sum_{i\geq3}n_i^2+1\leq \sum_{i\geq 3}(2+n_i^2)+5\leq2\sum_{i\geq3}\lfloor n_i\rfloor^2 .
\end{equation}
On the other hand,\ by (\ref{0007}),\ we obtain
\begin{equation}\label{0010}
\left|n_1+\left(\frac{\mu_2}{\mu_1}\right)n_2\right|\leq \sum_{i\geq 3}|n_i|+m^*(k,k').
\end{equation}
To prove the inequality (\ref{0006}),\ we will distinguish two cases:

\textbf{Case. 1.} $\frac{\mu_2}{\mu_1}=-1$.

\textbf{Subcase. 1.1.} $n_1=n_2$.

Then it is easy to show that
\begin{equation*}
\sum_{n\in\mathbb{Z}}|k_n-k_n'|\ln^{\sigma}\lfloor n\rfloor\leq \sum_{i\geq3}\ln^{\sigma}\lfloor n_i\rfloor\leq \sum_{i\geq3}\ln^{\sigma}\lfloor n_i\rfloor+\ln^{\sigma}\lfloor m^*(k,k')\rfloor.
\end{equation*}

\textbf{Subcase. 1.2.} $n_1\neq n_2$.

Then one has
\begin{eqnarray}
\nonumber|n_1-n_2|+|n_1+n_2|
&\leq&\nonumber|n_1-n_2|+|n_1^2-n_2^2|\\
& \leq&\nonumber\sum_{i\geq 3}|n_i|+m^*(k,k')+2\sum_{i\geq 3}\lfloor n_i\rfloor^2\\ \nonumber &&\quad \mbox{(in view of (\ref{0009}) and (\ref{0010}))}\\
\label{0011} &\leq&3\sum_{i\geq3}\lfloor n_i\rfloor^2+\lfloor m^{*}(k,k')\rfloor.
\end{eqnarray}
Hence one has
\begin{equation*}
\max\{|n_1|,|n_2|\}\leq \max\{|n_1-n_2|,|n_1+n_2|\}\leq 3\sum_{i\geq3}\lfloor n_i\rfloor^2+ \lfloor m^{*}(k,k')\rfloor^2,
\end{equation*}
and then
\begin{equation*}
\lfloor n_{1}\rfloor\leq 3\sum_{i\geq3}\lfloor n_i\rfloor^2+ \lfloor m^{*}(k,k')\rfloor^2.
\end{equation*}
For $j=1,2,$ one has
\begin{equation*}
\ln^{\sigma}\lfloor n_j\rfloor\leq \ln^{\sigma} \left(3\sum_{i\geq3}\lfloor n_i\rfloor^{2}+\lfloor m^{*}(k,k')\rfloor^2\right)\leq  2^{\sigma}\ln^{\sigma}\left(\sum_{i\geq3}\lfloor n_i\rfloor^{2}+\lfloor m^{*}(k,k')\rfloor^{2}\right).
\end{equation*}
Using (\ref{122201}) in Lemma \ref{122203} again and again, one has
\begin{equation}\label{0012}
\ln^{\sigma}\left(\sum_{i\geq3}\lfloor n_i\rfloor^{2}+\lfloor m^{*}(k,k')\rfloor^{2}\right)\leq 2^{\sigma}\left(\sum_{i\geq3}\ln^{\sigma}\lfloor n_i\rfloor+\ln^{\sigma}\lfloor m^{*}(k,k')\rfloor\right).
\end{equation}
Therefore we have
\begin{eqnarray}
\nonumber\sum_{n\in\mathbb{Z}}|k_n-k_n'|\ln^{\sigma}\lfloor n\rfloor
&\leq&\nonumber\sum_{n\in\mathbb{Z}}(k_n+k_n')\ln^{\sigma}\lfloor n\rfloor\\
&=&\nonumber\sum_{i\geq1}\ln^{\sigma}\lfloor n_i\rfloor\\
&\leq&\label{0013}3 \cdot 4^{\sigma}\left(\sum_{i\geq3}\ln^{\sigma}\lfloor n_i\rfloor+\ln^{\sigma}\lfloor m^{*}(k,k')\rfloor\right).
\end{eqnarray}
\textbf{Case. 2.} $\frac{\mu_2}{\mu_1}=1$.

In view of (\ref{0009}), one has
\begin{equation*}
n_1^2+n_2^2\leq 2\sum_{i\geq3}\lfloor n_i\rfloor^{2},
\end{equation*}which implies
\begin{equation*}
\lfloor n_j\rfloor \leq 2\sum_{i\geq3}\lfloor n_i\rfloor^{2},  \ \mbox{($j=1,2$)}.
\end{equation*}
Following the proof of (\ref{0013}), we have
\begin{equation*}
\sum_{n\in\mathbb{Z}}|k_n-k_n'|\ln^{\sigma}\lfloor n\rfloor \leq 3 \cdot 4^{\sigma}\left(\sum_{i\geq3}\ln^{\sigma}\lfloor n_i\rfloor+\ln^{\sigma}\lfloor m^{*}(k,k')\rfloor\right).
\end{equation*}
\end{proof}

\begin{lem}For $\sigma>2$ and $\delta\in(0,1)$, then we have
\begin{equation*}\label{042805}
\max_{x\geq0}\ \exp\left\{-\delta x^{\sigma}+x\right\}\leq \exp\left\{\left(\frac1{\delta}\right)^{\frac1{\sigma-1}}\right\}.
\end{equation*}
\end{lem}
\begin{proof}
	The proof of this lemma see the proof of Lemma 4.4 in \cite{cong2024}.
\end{proof}

\begin{lem}\label{8.6}For $p\geq 1$ and $\delta\in (0,1)$, then one has
\begin{equation}\label{042805*}
\max_{x\geq0}\ x^{p}e^{-\delta x}\leq \left(\frac{p}{e\delta}\right)^p.
\end{equation}
\end{lem}
\begin{proof}
	The proof of this lemma see the proof of Lemma 4.5 in \cite{cong2024}.
\end{proof}

\begin{lem}\label{lem2}
For $\sigma>2$ and $\delta\in(0,1)$, then we have
\begin{equation}\label{0418011}
\sum_{j\geq 1}e^{-\delta \ln^{\sigma} j}\leq \frac{6}\delta\cdot \exp\left\{\left(\frac1{\delta}\right)^{\frac1{\sigma-1}}\right\}.
\end{equation}
\end{lem}
\begin{proof}
	The proof of this lemma see the proof of Lemma 4.6 in \cite{cong2024}.
\end{proof}

\begin{lem}\label{a3}
For $\sigma>2$ and $\delta\in(0,1)$,  we have the following inequality
\begin{equation}\label{041809}
\sum_{a\in\mathbb{N}^{\mathbb{Z}}}e^{-\delta\sum_{ n\in\mathbb{Z}}a_{n}\ln^{\sigma}\lfloor n\rfloor}\leq\prod_{ n\in\mathbb{Z}}\frac{1}{1-e^{-\delta \ln^{\sigma}\lfloor n\rfloor}}.
\end{equation}
\end{lem}
\begin{proof}
The proof of this lemma see the proof of Lemma 4.3 in \cite{cong2024}.
\end{proof}

\begin{lem}\label{a5}
For $\sigma>2$ and $0<\delta\ll 1$, then we have
\begin{equation}\label{122401}
\prod_{ n\in\mathbb{Z}}\frac{1}{1-e^{-\delta \ln^{\sigma}\lfloor n\rfloor}}\leq\exp\left\{\left(\frac{18}{\delta}\right)\cdot \exp\left\{{\left(\frac{4}\delta \right)}^{\frac{1}{\sigma-1}}\right\}\right\}.
\end{equation}
\end{lem}
\begin{proof}
The proof of this lemma see the proof of Lemma 4.7 in \cite{cong2024}..
\end{proof}
\begin{lem}
\label{b1}For $\sigma>2$, $\delta\in(0,1)$, $p=1,2$ and $a=(a_{n})_{ n\in\mathbb{Z}}\in\mathbb{N}^{\mathbb{Z}^d}$, then we have
\begin{equation}\label{042807}
\prod_{n\in\mathbb{Z}}\left(1+a_{n}^p\right)e^{-2\delta a_{n}\ln^{\sigma} \lfloor n\rfloor}\leq \exp\left\{3p\left(\frac{p}{\delta}\right)^{\frac 1{\sigma-1}}\cdot\exp\left\{\left(\frac1\delta\right)^{\frac1\sigma}\right\}
\right\}.
\end{equation}
\end{lem}
\begin{proof}
The proof of this lemma see the proof of Lemma 4.8 in \cite{cong2024}.
\end{proof}

% You may incorporate your references as follows in your main tex file.
% Using BibTex is not recommended but can be handled.


\begin{thebibliography}{10}

\bibitem{BBM2014}
P.~Baldi, M.~Berti, and R.~Montalto.
\newblock K{AM} for quasi-linear and fully nonlinear forced perturbations of
  {A}iry equation.
\newblock {\em Math. Ann.}, 359(1-2):471--536, 2014.

\bibitem{BBM2016}
P.~Baldi, M.~Berti, and R.~Montalto.
\newblock K{AM} for autonomous quasi-linear perturbations of {K}d{V}.
\newblock {\em Ann. Inst. H. Poincar\'{e} Anal. Non Lin\'{e}aire},
  33(6):1589--1638, 2016.

\bibitem{B1996}
J.~Bourgain.
\newblock Construction of approximative and almost periodic solutions of
  perturbed linear {S}chr\"{o}dinger and wave equations.
\newblock {\em Geom. Funct. Anal.}, 6(2):201--230, 1996.

\bibitem{BA1998}
J.~Bourgain.
\newblock Quasi-periodic solutions of {H}amiltonian perturbations of 2{D}
  linear {S}chr\"{o}dinger equations.
\newblock {\em Ann. of Math. (2)}, 148(2):363--439, 1998.

\bibitem{BZ2004}
J.~Bourgain.
\newblock Recent progress in quasi-periodic lattice {S}chr\"{o}dinger operators
  and {H}amiltonian partial differential equations.
\newblock {\em Uspekhi Mat. Nauk}, 59(2(356)):37--52, 2004.

\bibitem{BFN2020}
A.~Bounemoura, B.~Fayad and L.~Niederman
\newblock Super-exponential stability for generic real-analytic elliptic
equilibrium points.
\newblock {\em Adv. Math.}, 366: 107088–1, 2020.



\bibitem{BJFA2005}
J.~Bourgain.
\newblock On invariant tori of full dimension for 1{D} periodic {NLS}.
\newblock {\em J. Funct. Anal.}, 229(1):62--94, 2005.

\bibitem{BMP2021}
L.~Biasco, J.~Massetti and M.~Procesi.
\newblock On the construction of Sobolev Almost periodic invariant tori for the 1d NLS.
\newblock {\em  Ann. Inst. H. Poincaré Anal. Non Linéaire}, 38(3), 711--758, 2021.

\bibitem{CMSW}
 Hongzi Cong, Lufang Mi,  Yunfeng Shi and Yuan Wu.
\newblock  On the existence of full dimensional KAM torus for nonlinear Schr\"odinger equation.
 \newblock {\em Discrete and Continuous Dynamical Systems-A}, 39(11):6599--6630,2019.

\bibitem{CLSY}
H.~Cong, J.~Liu, Y.~Shi, and X.~Yuan.
\newblock The stability of full dimensional {KAM} tori for nonlinear
  {S}chr\"{o}dinger equation.
\newblock {\em J. Differential Equations}, 264(7):4504--4563, 2018.

\bibitem{cong2024}
H.~Cong.
\newblock The Existence of Full Dimensional KAM tori for Nonlinear Schr\"{o}dinger equation.
 \newblock {\em Mathematische Annalen}, 390:671–-719, 2024.

\bibitem{CY2021}
H.~Cong and X.~Yuan.
\newblock  The existence of full dimensional invariant tori for 1-dimensional nonlinear
wave equation.
\newblock {\em Ann. Inst. H. Poincaré Anal. Non Linéaire}, 38(3):759–-786, 2021.


\bibitem{CW1993}
W.~Craig and C.~E. Wayne.
\newblock Newton's method and periodic solutions of nonlinear wave equations.
\newblock {\em Comm. Pure Appl. Math.}, 46(11):1409--1498, 1993.



\bibitem{G2012}
J.~Geng.
\newblock Invariant tori of full dimension for a nonlinear {S}chr\"{o}dinger
  equation.
\newblock {\em J. Differential Equations}, 252(1):1--34, 2012.


\bibitem{GX2013}
J.~Geng and X.~Xu.
\newblock Almost periodic solutions of one dimensional {S}chr\"{o}dinger
  equation with the external parameters.
\newblock {\em J. Dynam. Differential Equations}, 25(2):435--450, 2013.

\bibitem{KPB2003}
T.~Kappeler and J.~P\"{o}schel.
\newblock {\em Kd{V} }\&{\em {KAM}}, volume~45 of {\em Ergebnisse der Mathematik und
  ihrer Grenzgebiete. 3. Folge.}
\newblock Springer-Verlag, Berlin, 2003.

\bibitem{K1987}
S.~B. Kuksin.
\newblock Hamiltonian perturbations of infinite-dimensional linear systems with
  imaginary spectrum.
\newblock {\em Funktsional. Anal. i Prilozhen.}, 21(3):22--37, 95, 1987.

\bibitem{KB2000}
S.~B. Kuksin.
\newblock { Analysis of {H}amiltonian {PDE}s}, volume~19 of {\em Oxford
  Lecture Series in Mathematics and its Applications}.
\newblock Oxford University Press, Oxford, 2000.

\bibitem{KZ2004}
S.~B. Kuksin.
\newblock Fifteen years of {KAM} for {PDE}.
\newblock In {\em Geometry, topolny, and mathematical physics}, volume 212 of
  {\em Amer. Math. Soc. Transl. Ser. 2}, pages 237--258. Amer. Math. Soc.,
  Providence, RI, 2004.

\bibitem{NG2007}
H.~Niu and J.~Geng.
\newblock Almost periodic solutions for a class of higher-dimensional beam
  equations.
\newblock {\em Nonlinearity}, 20(11):2499--2517, 2007.

\bibitem{P1990}
J.~P\"{o}schel.
\newblock Small divisors with spatial structure in infinite-dimensional
  {H}amiltonian systems.
\newblock {\em Comm. Math. Phys.}, 127(2):351--393, 1990.

\bibitem{P2002}
J.~P\"{o}schel.
\newblock On the construction of almost periodic solutions for a nonlinear
  {S}chr\"{o}dinger equation.
\newblock {\em Ergodic Theory Dynam. Systems}, 22(5):1537--1549, 2002.

\bibitem{W1990}
C.~E. Wayne.
\newblock Periodic and quasi-periodic solutions of nonlinear wave equations via
  {KAM} theory.
\newblock {\em Comm. Math. Phys.}, 127(3):479--528, 1990.

\bibitem{WG2011}
J.~Wu and J.~Geng.
\newblock Almost periodic solutions for a class of semilinear quantum harmonic
  oscillators.
\newblock {\em Discrete Contin. Dyn. Syst.}, 31(3):997--1015, 2011.


\end{thebibliography}
\end{document}